\documentclass[11pt,reqno]{amsart}
\usepackage{graphicx}
\usepackage{verbatim}
\usepackage[usenames,dvipsnames]{color}

\newtheorem{thm}{Theorem}[section]
\newtheorem{lem}[thm]{Lemma}
\newtheorem{prop}[thm]{Proposition}

\newtheorem{defn}[thm]{Definition}

\theoremstyle{remark}

\numberwithin{equation}{section}

\def\N{\mathbb{N}}

\def\R{\mathbb{R}}

\def\ra{\rightarrow}

\def\al{\alpha}
\def\be{\beta}
\def\ep{\epsilon}
\def\ka{\kappa}

\def\la{\lambda}

\def\si{\sigma}

\def\om{\omega}
\def\Om{\Omega}
\def\de{\delta}

\def\vp{\varphi}

\begin{document}

\title[Transition fronts]{Transition Fronts in Time Heterogeneous and Random Media of Ignition Type}

\author{Wenxian Shen}
\address{Department of Mathematics and Statistics, Auburn University, Auburn, AL 36849, USA}
\email{wenxish@auburn.edu}

\author{Zhongwei Shen}
\address{Department of Mathematics and Statistics, Auburn University, Auburn, AL 36849, USA}
\curraddr{Department of Mathematical and Statistical Sciences, University of Alberta, Edmonton, AB T6G 2G1, Canada}
\email{zhongwei@ualberta.ca\\zzs0004@auburn.edu}

\subjclass[2010]{35C07, 35K55, 35K57, 92D25}



\keywords{transition front, time heterogeneous media, time random media.}

\begin{abstract}
The current paper is devoted to the investigation of wave propagation phenomenon in reaction-diffusion equations with ignition type nonlinearity in time heterogeneous and random media. It is proven that such equations in time heterogeneous media admit transition fronts  with time dependent profiles and that such equations in time random media admit transition fronts with random profiles. Important properties of transition fronts, including the boundedness of propagation speeds and the uniform decaying estimates of the propagation fronts, are also obtained.
\end{abstract}

\maketitle

\tableofcontents


\section{Introduction}\label{sec-intro}

Consider the one-dimensional reaction-diffusion equation
\begin{equation}\label{general-eqn}
u_{t}=u_{xx}+f(t,x,u), \quad x\in\R,\,\,t\in\R,
\end{equation}
where $f(t,x,u)$ is of ignition type, that is, there exists $\theta\in(0,1)$ such that for all $t\in\R$ and $x\in\R$, $f(t,x,u)=0$ for $u\in[0,\theta]\cup\{1\}$ and $f(t,x,u)>0$ for $u\in(\theta,1)$. Such an equation arises in the combustion theory (see e.g. \cite{BeLaLi90,BeNiSc85}). The number $\theta$ is usually referred to as the ignition temperature. The front propagation concerning this equation was first investigated by Kanel (see \cite{Ka60,Ka61,Ka62,Ka64}) in the space-time homogeneous media, i.e., $f(t,x,u)=f(u)$; he proved that all solutions, with initial data in some subclass of continuous functions with compact support and values in $[0,1]$, propagate at the same speed $c_{*}>0$, which is the speed of the unique traveling wave solution $\psi(x-c_{*}t)$, where $\psi$ satisfies
\begin{equation*}
\begin{split}
&\psi_{xx}+c_{*}\psi_{x}+f(\psi)=0,\quad\lim_{x\ra-\infty}\psi(x)=1\quad\text{and}\quad\lim_{x\ra\infty}\psi(x)=0.
\end{split}
\end{equation*}
Also see \cite{ArWe75,ArWe78,FiMc77,FiMc80} and references therein for the treatment of traveling wave solutions of \eqref{general-eqn} in space-time homogeneous media.

Recently, equation \eqref{general-eqn} in the space heterogeneous media, i.e., $f(t,x,u)=f(x,u)$, has attracted a lot of attention. In terms of space periodic media, that is, $f(x,u)$ is periodic in $x$, Berestycki and Hamel proved in \cite{BeHa02} the existence of pulsating fronts or periodic traveling waves of the form $\psi(x-c_{*}t,x)$, where $\psi(s,x)$ is periodic in $x$ and satisfies a degenerate elliptic equation with boundary conditions $\lim_{s\ra-\infty}\psi(s,x)=1$ and $\lim_{s\ra\infty}\psi(s,x)=0$ uniformly in $x$. In the work of Weinberger (see \cite{We02}), he proved from the dynamical system viewpoint that solutions with general non-negative compactly supported initial data spread with the speed $c_{*}$. We also refer to \cite{Xin91,Xin92,Xin93} for related works.

In the general space heterogeneous media, wavefront with a profile is no longer appropriate, and we are looking for more general wavefronts such as transition fronts in the sense of Berestycki and Hamel (see \cite{BeHa07,BeHa12}), that is,

\begin{defn}\label{def-transition-wave}
A global-in-time solution $u(t,x)$, $x\in\R$, $t\in\R$ of \eqref{general-eqn} is called a {\rm transition front} if there is a function $\xi:\R\ra\R$, called, {\rm interface location function}, such that
\begin{equation*}
\begin{split}
&u(t,x)\ra1\,\,\text{uniformly in}\,\,t\,\,\text{and}\,\,x\leq\xi(t)\,\,\text{as}\,\,x-\xi(t)\ra-\infty,\,\,and\\
&u(t,x)\ra0\,\,\text{uniformly in}\,\,t\,\,\text{and}\,\,x\geq\xi(t)\,\,\text{as}\,\,x-\xi(t)\ra\infty.
\end{split}
\end{equation*}

A transition front $u(t,x)$ is called {\rm critical} if for any transition front $\tilde{u}(t,x)$ there exists a function $\zeta:\R\ra\R$ such that
\begin{equation*}
\begin{split}
u(t,x)&\geq\tilde{u}(t,x),\quad x\leq\zeta(t),\\
u(t,x)&\leq\tilde{u}(t,x),\quad x\geq\zeta(t)
\end{split}
\end{equation*}
for all $t\in\R$.
\end{defn}

Transition fronts are proper generalizations of traveling waves in homogeneous media and periodic traveling waves (or pulsating fronts) in periodic media. It is easily seen that the interface location function $\xi(t)$ in Definition \ref{def-transition-wave} is unique up to addition by bounded functions. Transition fronts in the above sense are also called {\it generalized traveling waves} in some literature, especially in the time heterogeneous media, i.e., $f(t,x,u)=f(t,u)$ (see \cite{Sh11-1}). Roughly speaking, critical transition fronts are the steepest ones among all transition fronts. It is known that the existence of a transition front implies the existence of a critical transition front and critical transition fronts (if exist) are unique up to phase shift (see Lemmas \ref{lm-critical-traveling-wave1} and Lemma \ref{lm-critical-traveling-wave2}).

In the work of Nolen and Ryzhik (see \cite{NoRy09}), and Mellet, Roquejoffre and Sire (see \cite{MeRoSi10}), transition fronts with additional properties, such as, time monotonicity, finite speed, exponential decay ahead of the interface, etc.,  are proven to exist in the general space heterogeneous media
of ignition type (the work \cite{NoRy09} also deals with transition fronts in space random media of ignition type). Later, stability and uniqueness of such transition fronts are also established in \cite{MNRR09}. These results are then generalized by Zlato\v{s} (see \cite{Zl13}) to the equations in space heterogeneous cylindrical domains of ignition type.

However, there is little study of transition fronts in general time heterogeneous and random  media of ignition type.
In the current paper, we first study front propagation phenomenon in the reaction-diffusion equation \eqref{general-eqn} in general time heterogeneous media, that is,
\begin{equation}\label{main-eqn}
u_{t}=u_{xx}+f(t,u), \quad x\in\R,\,\,t\in\R.
\end{equation}
Here are the assumptions on $f(t,u)$:
\begin{itemize}
\item[\rm(H1)] There is a $\theta\in(0,1)$, called the ignition temperature, such that for all $t\in\R$,
\begin{equation*}
\begin{split}
f(t,u)&=0,\quad u\in(-\infty,\theta]\cup\{1\},\\
f(t,u)&>0,\quad u\in(\theta,1),\\
f(t,u)&<0,\quad u>1.
\end{split}
\end{equation*}
The family of functions $\{f(t,u), u\in\R\}$ is locally uniformly H\"{o}lder continuous. The family of functions $\{f(t,u), t\in\R\}$ is locally uniformly Lipschitz continuous. For any $t\in\R$, $f(t,u)$ is continuously differentiable for $u\geq\theta$.
\item[\rm(H2)] There are Lipschitz continuous functions $f_{\inf}$, $f_{\sup}$ satisfying
\begin{equation*}
\begin{split}
&f_{\inf}, f_{\sup}\in C^1([\theta,\infty),\R)\\
&f_{\inf}(u)=0=f_{\sup}(u)\,\,\text{for}\,\,u\in[0,\theta]\cup\{1\},\\
&0>(f_{\inf})_{u}(1)\geq(f_{\sup})_{u}(1),\\
&0<f_{\inf}(u)<f_{\sup}(u)\,\,\text{for}\,\,u\in(\theta,1)
\end{split}
\end{equation*}
such that $f_{\inf}(u)\leq f(t,u)\leq f_{\sup}(u)$ for $u\in[\theta,1]$ and $t\in\R$.
\end{itemize}

We prove

\begin{thm}\label{thm-transition-wave}
Suppose $\rm(H1)$ and $\rm(H2)$.
 \begin{itemize}
 \item[\rm(1)] (Existence of transition front)  Equation \eqref{main-eqn} admits a transition front $u(t,x)$, $x\in\R$, $t\in\R$ in the sense of Definition \ref{def-transition-wave}, where the  function $\xi:\R\ra\R$ is continuously differentiable and satisfies $u(t,\xi(t))=\theta$ for all $t\in\R$. Moreover, the following properties hold:
\begin{itemize}
\item[\rm(i)] (Monotonicity of the transition front) $u_{x}(t,x)<0$ for $x\in\R$ and $t\in\R$;
\item[\rm(ii)] (Uniform steepness) $\sup_{t\in\R}u_{x}(t,\xi(t))<0$;
\item[\rm(iii)] (Finite speed) $\sup_{t\in\R}|\xi'(t)|<\infty$;
\item[\rm(iv)]  (Uniform decaying estimates) there exists a continuous and strictly decreasing function $v:\R\ra(0,1)$ satisfying $v(x)\geq 1-c_{1}e^{c_{2}x}$, $x\leq-c_{3}$ for some $c_{1},c_{2},c_{3}>0$ and $v(x)=\theta e^{-cx}$, $x\geq0$ for some $c>0$ such that
\begin{equation*}
\begin{split}
u(t,x+\xi(t))&\geq v(x),\quad x\leq0;\\
u(t,x+\xi(t))&\leq v(x),\quad x\geq0.
\end{split}
\end{equation*}
\end{itemize}

\item[\rm(2)] (Periodicity) If $f(t+T,u)=f(t,u)$, then \eqref{main-eqn} admits a periodic traveling wave $u(t,x)$, that is, there are a constant $c\in\R$ and a function $\psi:\R\times\R\ra(0,1)$ satisfying
\begin{equation*}\label{eqn-periodic-tw}
\begin{cases}
\psi_{t}=\psi_{xx}+c\psi_{x}+f(t,\psi),\cr
\lim_{x\ra-\infty}\psi(t,x)=1,\,\,\lim_{x\ra\infty}\psi(t,x)=0\,\,\text{uniformly in}\,\,t\in\R,\cr
\psi(t,\cdot)=\psi(t+T,\cdot)\,\,\text{for all}\,\,t\in\R
\end{cases}
\end{equation*}
such that $u(t,x)=\psi(t,x-ct)$ for $x\in\R$ and $t\in\R$.
\end{itemize}
\end{thm}

Clearly, due to the space homogeneity of \eqref{main-eqn}, if $u(t,x)$ is a transition front of \eqref{main-eqn}, then any space translation of $u(t,x)$ is also a transition front. All these consists of a family of transition fronts propagating to the right. By space reflection, we obtain another family propagating to the left. We see that the transition front constructed in Theorem \ref{thm-transition-wave} has a time-dependent profile given by $\psi(t,x)=u(t,x+\xi(t))$, which is a solution of
\begin{equation}\label{eqn-time-profile}
\begin{cases}
\psi_{t}=\psi_{xx}+\xi'(t)\psi_{x}+f(t,\psi),\cr
\lim_{x\ra-\infty}\psi(t,x)=1,\,\,\lim_{x\ra\infty}\psi(t,x)=0\,\,\text{uniformly in}\,\,t\in\R.
\end{cases}
\end{equation}

We then study front propagation phenomena in reaction-diffusion equations in random media, that is,
\begin{equation}
\label{random-eq}
u_t=u_{xx}+f(\sigma_{t}\omega,u),\quad x\in \R, \,\, t\in\R,
\end{equation}
where $\omega\in\Omega$, $((\Omega,\mathcal{F},P),\{\sigma_{t}\}_{t\in\R})$ is a metric dynamical system
(i.e. $(\Omega,\mathcal{F},P)$ is a probability space, the mapping $(t,\omega)\mapsto \sigma_t(\omega):\R\times\Omega\ra\R$ is measurable,
 $\sigma_t\circ\sigma_s=\sigma_{t+s}$ for any $s,t\in\R$, and $P(\sigma_t F)=P(F)$ for any $t\in\R$ and $F\in\mathcal{F}$) and $f:\Omega\times\R\to\R$ satisfies

\begin{itemize}
\item[\rm(H3)] $f:\Omega\times\R\to\R$ is measurable, and for each $\omega\in\Omega$, $f^\omega(t,u)=f(\theta_{t}\omega,u)$ satisfies $\rm(H1)$ and $\rm(H2)$.
\end{itemize}

We look for random traveling wave solutions of \eqref{random-eq} in the following sense (see \cite{Sh04}).

\begin{defn}
\label{random-wave-def} A family $\{u(t,x;\omega)\}_{\omega\in\Omega}$ of global-in-time solutions of \eqref{random-eq} is called a
{\rm random traveling wave} if there are measurable functions $\Psi:\R\times\Om\to\R$ and $\xi:\R\times\Om\to \R$ such that
$$
u(t,x;\omega)=\Psi(x-\xi(t;\omega),\sigma_t \omega),
$$
and for each fixed $\omega\in\Omega$, $u(t,x;\omega)$ is a transition front of \eqref{random-eq}, that is,
$$
\lim_{x\to -\infty}\Psi(x,\sigma_t \omega)=1,\quad \lim_{x\to\infty}\Psi(x,\sigma_t\omega)=0
$$
uniformly in $t\in\R$.
\end{defn}

We prove

\begin{thm}
\label{thm-random-wave} Assume $\rm(H3)$.
\begin{itemize}
\item[\rm(1)] Equation \eqref{random-eq} admits a random traveling wave $u(t,x;\omega)=\Psi(x-\xi(t;\omega),\sigma_t \omega)$, where the  function $\xi:\R\times \Omega\ra\R$ is continuously differentiable in $t\in\R$ and satisfies $u(t,\xi(t,\omega);\omega)=\theta$ for all $t\in\R$ and $\omega\in\Omega$.

\item[\rm(2)] If $((\Omega,\mathcal{F},P),\{\sigma_t\}_{t\in\R})$ is an ergodic metric dynamical system, then there are $c^*\in\R$ and $\Psi^*(\cdot)\in C_{\rm unif}^b (\R,\R)$ such that for a.e.
$\omega\in\Omega$,
$$
\lim_{t\to\infty}\frac{\xi(t;\omega)}{t}=c^{*},
$$
$$
\lim_{t\to\infty}\frac{1}{t}\int_0^ t\Psi(x,\sigma_s\omega)ds=\Psi^*(x),\quad \forall \,x\in\R,
$$
and
$$
\lim_{x\to -\infty}\Psi^*(x)=1,\quad \lim_{x\to\infty} \Psi^*(x)=0.
$$

\item[\rm(3)]  If $((\Om,\mathcal{F},P),\{\sigma_t\}_{t\in\R})$ is a compact flow, then there is $\Om_0\subset\Om$ with $\si_t(\Om_0)=\Om_0$ such that $\Om_0$ is a residual subset of $\Om$ (i.e., $\Om_{0}$ is the intersection of countably many open dense subsets of $\Om$) and  the map $\om\in\Om\mapsto \Psi(\cdot,\om)\in C_{\rm unif}^b(\R,\R)$ is continuous at $\om\in\Om_0$.
\end{itemize}
\end{thm}

We remark that if $\Omega=\{f(\cdot+\tau,\cdot)|\tau\in\R\}$ equipped with open compact topology and $\sigma_t g(\cdot,\cdot)=g(\cdot+t,\cdot)$ for $g\in\Omega$ and $t\in\R$, where $f(t,u)$ satisfies (H1), (H2) and $f(t+T,u)=f(t,u)$,
then Theorem \ref{thm-random-wave} (3) implies that \eqref{main-eqn} admits a periodic traveling wave solution, which recovers Theorem \ref{thm-transition-wave} (2). In the case that $\Omega={\rm hull}(f):={\rm cl}\{f(\cdot+\tau,\cdot)|\tau\in\R\}$ with open compact topology and $\sigma_t g(\cdot,\cdot)=g(\cdot+t,\cdot)$ for $g\in \Omega$ and $t\in\R$, where $f(t,u)$ satisfies $\rm(H1)$ and $\rm(H2)$ and is almost periodic in $t$ uniformly with respect to $u$, whether \eqref{main-eqn} admits almost periodic traveling wave solutions remains open. This issue together with the uniqueness and stability of transition fronts of \eqref{main-eqn} are studied in \cite{SS-1}.

We also remark that time-periodic traveling waves were first investigated by Alikakos, Bates and Chen (see \cite{AlBaCh99}) in time periodic bistable media. For time heterogeneous bistable equations, transition fronts with a time-dependent profile satisfying \eqref{eqn-time-profile} and their uniqueness and stability have been investigated by Shen (see e.g. \cite{Sh99-1,Sh99-2,Sh04,Sh06}). There are  similar results for time heterogeneous KPP equations (see e.g. \cite{NaRo12,Sh11,RR14}). Transition fronts have also been proven to exist in space heterogeneous Fisher-KPP type equations (see \cite{NRRZ12,Zl12}). But it is far from being clear in space heterogeneous media of bistable type due to the wave blocking phenomenon (see \cite{LeKe00}) except the one established in \cite{NoRy09} under additional assumptions.

To this end, we comment on the differences  between the analysis in the present paper and that in \cite{Sh06} and \cite{NoRy09}. In \cite{Sh06}, transition fronts of the following equation
\begin{equation}\label{eqn-time-hetero-bistable}
u_{t}=u_{xx}+f_{B}(t,u),\quad (t,x)\in\R\times\R
\end{equation}
in time heterogeneous bistable case were studied. In particular,  $u\equiv0$ and $u\equiv1$ are two stable solutions and $\theta(t)$, $t\in\R$ is the unstable solution between $0$ and $1$. The method used in \cite{Sh06} has a strong dynamical system favor. More precisely, instead of focusing on \eqref{eqn-time-hetero-bistable}, the following family
$$
u_{t}=u_{xx}+f_{B}(t+s,u),\quad (t,x)\in\R\times\R,\quad s\in\R
$$
were treated as a whole. Since the analysis in \cite{Sh06} heavily relies on the uniform instability of the solution $\theta(t)$, the method can not be applied in our case.

In \cite{NoRy09}, transition fronts in space-heterogeneous ignition equations were treated by studying the following equation
\begin{equation}\label{eqn-time-hetero-ignition}
u_{t}=u_{xx}+f_{I}(x,u),\quad (t,x)\in\R\times\R
\end{equation}
with the neutral stable solution $u\equiv0$, the stable solution $u\equiv1$ and the ignition temperature $\theta\in(0,1)$. The proof of the existence of transition fronts is constructive via the construction of approximating solution $u^{n}(t,x)$, where $u^{n}(t,x)$ is the unique solution of \eqref{eqn-time-hetero-ignition} with well-constructed initial data at initial time $t=-n$. An important property of $u^{n}(t,x)$ is the time monotonicity, i.e., $u^{n}_{t}(t,x)>0$, which implies $\dot{\xi}^{n}(t)>0$, where $\xi^{n}(t)$ is the interface location defined by
\begin{equation*}
\xi^{n}(t)=\sup\{x\in\R|u^{n}(t,x)=\theta\}.
\end{equation*}
The fact that $\xi^{n}(t)$ is increasing plays a very important role in the analysis done in \cite{NoRy09}.

In our case, we first construct approximating solutions $u^{n}(t,x)$ as in \cite{NoRy09}, but our approximating solutions satisfy space monotonicity, i.e., $u_{x}^{n}(t,x)<0$, instead of time monotonicity. We then look at the interface location $\xi^{n}(t)$ defined to be the unique point such that $u^{n}(t,\xi^{n}(t))=\theta$. However, due to the time-dependence of the nonlinearity $f(t,u)$, $\xi^{n}(t)$ oscillates, and therefore, the analysis in \cite{NoRy09} does not apply. A major part of the present paper is devoted to the analysis of the propagation of $\xi^{n}(t)$ with oscillations.

The rest of the paper is organized as follows. In Section \ref{sec-entire-sol}, we construct a global-in-time solution of \eqref{main-eqn} as the limit of approximating solution sequence. Section \ref{sec-bd-interface-width} is devoted to the boundedness of interface width. In Section \ref{Non-Flat Estimate Near Ignition Temperature}, we prove that the derivative of the approximating solution sequence near the ignition temperature is uniformly negative. In Section \ref{uniform-estimate}, we establish uniform estimates behind and ahead of the interface for the approximating solution sequence. In Section \ref{sec-transition-wave-const}, we prove the existence of transition fronts of \eqref{main-eqn} and finish the proof Theorem \ref{thm-transition-wave}. In Section \ref{sec-random-wave}, we investigate  random traveling wave solutions of \eqref{random-eq} and prove Theorem \ref{thm-random-wave}.


\section{Construction of Global-in-Time Solutions}\label{sec-entire-sol}

In this section, we construct a global-in-time solution of \eqref{main-eqn}. Throughout this section, we assume (H1) and (H2).

First, we consider the space-time homogeneous equation
\begin{equation}\label{eqn-ig-inf}
u_{t}=u_{xx}+f_{\inf}(u).
\end{equation}
By $\rm(H2)$, $f_{\inf}$ is of standard ignition type. Classical results  (see e.g. \cite{ArWe75,ArWe78,FiMc77}) ensure the existence of a unique constant $c_{\inf}>0$ and a twice continuously differentiable function $\phi$ satisfying
\begin{equation}\label{traveling-homo}
\begin{cases}
\phi_{xx}+c_{\inf}\phi_{x}+f_{\inf}(\phi)=0,\cr
\phi_{x}<0,\quad\lim_{x\ra-\infty}\phi(x)=1\,\,\text{and}\,\,\lim_{x\ra\infty}\phi(x)=0
\end{cases}
\end{equation}
such that $\phi(x-c_{\inf}t)$, $x\in\R$, $t\in\R$ and its translations are traveling wave solutions of \eqref{eqn-ig-inf}. Thus, we may assume, without loss of generality, that $\phi(0)=\theta$. Since $f_{\inf}(u)=0$ for $u\in[0,\theta]$, a direct computation gives
\begin{equation}\label{explicit-sol}
\phi(x)=\theta e^{-c_{\inf}x},\quad x\geq0.
\end{equation}

\begin{lem}\label{lem-entire-sol}
For any $s<0$, there is a unique $x_{s}\in\R$ such that the solution $u(t,x;s)$, $t\geq s$ of \eqref{main-eqn} with $u(s,x;s)=\phi(x-x_{s})$ satisfies $u(0,0;s)=\theta$. Moreover, $x_{s}\ra-\infty$ as $s\ra-\infty$.
\end{lem}
\begin{proof}
Fix any $s<0$. Let $u^{y}(t,x;s)$ be the solution of \eqref{main-eqn} with $u^{y}(s,x;s)=\phi(x-y)$. By comparison principle, $u^{y}(t,x;s)\geq\phi(x-y-c_{\inf}(t-s))$ for $t\geq s$. In particular, $u^{y}(0,0;s)\geq\phi(c_{\inf}s-y)$. Note that if $y>c_{\inf}s$, then $\phi(c_{\inf}s-y)>\phi(0)=\theta$ by monotonicity. Thus, $u^{y}(0,0;s)>\theta$ if $y>c_{\inf}s$.

On the other hand, let us fix some constant $M>0$ such that $f(t,u)\leq Mu$ for all $u\geq0$ and $t\in\R$. Such an $M$ exists by $\rm(H1)$ and $\rm(H2)$. Now, set $v^{y}(t,x;s)=e^{-c_{\inf}(x-y-y_{0}-c(t-s))}$ for some $y_{0}\in\R$ and $c>0$ to be chosen. By \eqref{explicit-sol}, we can easily find an $y_{0}\in\R$ such that $\phi(x-y)\leq v^{y}(s,x;s)$ for all $y\in\R$. We fix such an $y_{0}$. We compute
\begin{equation*}
(v^{y})_{t}-(v^{y})_{xx}-f(t,v^{y})=c_{\inf}(c-c_{\inf})v^{y}-f(t,v^{y}).
\end{equation*}
Thus, if we choose $c>0$ such that $c_{\inf}(c-c_{\inf})\geq M$, then $v^{y}$ is a sup-solution of \eqref{main-eqn}, which leads to $u^{y}(t,x;s)\leq v^{y}(t,x;s)=e^{-c_{\inf}(x-y-y_{0}-c(t-s))}$ by comparison principle. In particular, $u^{y}(0,0;s)\leq e^{-c_{\inf}(cs-y-y_{0})}$. Thus, $u^{y}(0,0;s)<\theta$ for $y\ll-1$. Continuity of the solution with respect to $y$ then ensures the existence of some $x_{s}$ as in the statement of the lemma.

The uniqueness follows from comparison principle. In fact, if there are $x_{s}$ and $x_{s}^{*}$ with $x_{s}\neq x_{s}^{*}$, then we have either $u^{x_{s}}(0,x;s)<u^{x_{s}^{*}}(0,x;s)$ or $u^{x_{s}^{*}}(0,x;s)<u^{x_{s}}(0,x;s)$ for all $x\in\R$ by comparison principle, since either $\phi(x-x_{s})<\phi(x-x_{s}^{*})$ or $\phi(x-x_{s}^{*})<\phi(x-x_{s})$ holds for all $x\in\R$. Hence, for different $x_{s}$ and $x_{s}^{*}$, we can not have both $u^{x_{s}}(0,0;s)=\theta$ and $u^{x_{s}^{*}}(0,0;s)=\theta$.

The ``moreover" part is a simple consequence of the estimate
\begin{equation}\label{estimate-lower-bound}
u^{x_{s}}(t,x;s)\geq\phi(x-x_{s}-c_{\inf}(t-s)).
\end{equation}
In fact, if $\inf_{s<0}x_{s}>-\infty$, then for all $s\ll0$, $u^{x_{s}}(0,x;s)\geq\phi(0-x_{s}-c_{\inf}(0-s))>\theta$. It is a contradiction.
\end{proof}

From the above lemma, we can construct a global-in-time solution.

\begin{thm}\label{thm-entire-sol}
There exists a sequence $\{s_{n}\}_{n\in\N}\subset(-\infty,0)$ with $s_{n}\ra-\infty$ as $n\ra\infty$ and a function $u(t,x)$, $x\in\R$, $t\in\R$ continuously differentiable in $t$ and twice continuously differentiable in $x$ such that for any compact $K\subset\R\times\R$, the following limits
\begin{equation*}
\begin{split}
\lim_{n\ra\infty}u(t,x;s_{n})&=u(t,x),\quad\lim_{n\ra\infty}u_{t}(t,x;s_{n})=u_{t}(t,x),\\
\lim_{n\ra\infty}u_{x}(t,x;s_{n})&=u_{x}(t,x),\quad\lim_{n\ra\infty}u_{xx}(t,x;s_{n})=u_{xx}(t,x)
\end{split}
\end{equation*}
exist and are uniform in $(t,x)\in K$. In particular, $u(t,x)$, $x\in\R$, $t\in\R$ is a global-in-time solution of \eqref{main-eqn}.
\end{thm}
\begin{proof}
It is a consequence of Lemma \ref{lem-entire-sol}, a priori estimates for parabolic equations (see e.g. \cite{Fri64}), Arzela-Ascoli theorem and the diagonal argument.
\end{proof}

The global-in-time solution $u(t,x)$ constructed in Theorem \ref{thm-entire-sol} is a candidate for the expected transition front. All we need is to show that this solution satisfies certain non-degenerate and uniform decaying estimates. This, however, can be deduced from the boundedness of interface width, the steepness estimate, and the uniform decaying estimates of the approximating solutions $u(t,x;s)$, which are the objectives of Section \ref{sec-bd-interface-width}, Section \ref{Non-Flat Estimate Near Ignition Temperature} and Section \ref{uniform-estimate}, respectively.  In Section \ref{sec-transition-wave-const}, we finish the construction of transition fronts.

In the rest of this section, we derive some fundamental properties of $u(t,x;s)$.

\begin{lem}\label{lem-basic-prop}
For any $s<0$ and $t\geq s$, there hold the following properties:
\begin{equation*}
\lim_{x\ra-\infty}u(t,x;s)=1,\quad\lim_{x\ra\infty}u(t,x;s)=0\quad\text{and}\quad u_{x}(t,x;s)<0.
\end{equation*}
\end{lem}
\begin{proof}
The limit at $+\infty$ follows from the following two-sided estimates
\begin{equation}\label{aprior-estimate-two-side1}
\phi(x-x_{s}-c_{\inf}(t-s))\leq u(t,x;s)\leq e^{-c_{\inf}(x-x_{s}-y_{0}-c(t-s))},
\end{equation}
where the lower bound and the upper bound are constructed in Lemma \ref{lem-entire-sol}. The limit at $-\infty$ follows from the following two-sided estimates
\begin{equation*}\label{aprior-estimate-two-side11}
\phi(x-x_{s}-c_{\inf}(t-s))\leq u(t,x;s)\leq 1,
\end{equation*}
where the lower bound is constructed in Lemma \ref{lem-entire-sol} and the upper bound is due to the fact that $u\equiv 1$ is a solution of \eqref{main-eqn}
and $u(s,x;s)<1$ for all $x\in\R$.

We now show $u_{x}(t,x;s)<0$. Clearly, it is the case if $t=s$. So we assume $t>s$. Since $\phi(x-x_{s})$ is strictly decreasing, we apply maximum principle to $u(t,x+y;s)-u(t,x;s)$ for any $y>0$ to conclude that $u(t,x+y;s)<u(t,x;s)$. That is, $u(t,x;s)$ is strictly decreasing. For contradiction, suppose $u_{x}(t_{0},x_{0};s)=0$ for some $t_{0}>s$ and $x_{0}\in\R$. Let $u^{o}(t;t_{0},a)$ be the solution of the ODE $u_{t}=f(t,u)$ with $u^{o}(t_{0};t_{0},a)=a=u(t_{0},x_{0};s)$. Note $u^{o}(t;t_{0},a)$ extends naturally for $t<t_{0}$. Let $v(t,x;s)=u(t,x;s)-u^{o}(t;t_{0},a)$. It satisfies $v(t_{0},x_{0};s)=0$, $v_{x}(t_{0},x_{0};s)=0$ and the linear equation
\begin{equation}\label{eqn-linear-zero}
v_{t}=v_{xx}+q(t,x)v,
\end{equation}
where
\begin{equation*}
\begin{split}
q(t,x)=\left\{\begin{aligned}
\frac{f(t,u(t,x;s))-f(t,u^{o}(t;t_{0},a))}{u(t,x;s)-u^{o}(t;t_{0},a)},&\quad u(t,x;s)\neq u^{o}(t;t_{0},a),\\
0,&\quad u(t,x;s)=u^{o}(t;t_{0},a)
\end{aligned} \right.
\end{split}
\end{equation*}
is bounded. Applying Angenent's result (see e.g. \cite[Theorem B]{Ang88}) to \eqref{eqn-linear-zero}, there exist $\ep>0$ and $\de>0$ such that $v(t-\de,x;s)$ has at least two zeros in the interval $[x_{0}-\ep,x_{0}+\ep]$. However, due to the monotonicity of $u(t-\de,x;s)$ in $x$, $v(t-\de,x;s)$ has exactly one zero. This is a contradiction. Hence, $u_{x}(t,x;s)<0$.
\end{proof}

By Lemma \ref{lem-basic-prop}, for any $\la\in(0,1)$, $s<0$ and $t\geq s$, there is a unique $\xi_{\la}(t;s)\in\R$ such that
\begin{equation*}
u(t,\xi_{\la}(t;s);s)=\la.
\end{equation*}
The case $\la=\theta$ is of particular interest and it does play an important role in our later arguments. Notice $\xi_{\theta}(s;s)=x_{s}$ for all $s<0$. As usual, we refer to the point $(\xi_{\la}(t;s),\la)$ on the solution curve as the interface and $\xi_{\la}(t;s)$ as the interface location. The following lemma shows the continuous differentiability of $\xi_{\la}(t,s)$ in $t$.

\begin{lem}\label{lem-negativity}
Let $\la\in(0,1)$. For any $s<0$, the interface location $\xi_{\la}(t,s)$ is continuously differentiable in $t$ for $t>s$. Moreover, there holds
\begin{equation*}
\frac{d\xi_{\la}(t;s)}{dt}=-\frac{u_{t}(t,\xi_{\la}(t;s);s)}{u_{x}(t,\xi_{\la}(t;s);s)}.
\end{equation*}
\end{lem}
\begin{proof}
The continuity follows from the continuity of $u(t,x;s)$ and its monotonicity in $x$ by Lemma \ref{lem-basic-prop}. We show the continuous differentiability. Since $u(t,\xi_{\la}(t;s);s)=\la$ for $t\geq s$, we have $u(t+\ep,\xi_{\la}(t+\ep;s);s)-u(t,\xi_{\la}(t;s);s)=0$. Thus,
\begin{equation*}
\begin{split}
&\frac{u(t,\xi_{\la}(t+\ep;s);s)-u(t,\xi_{\la}(t;s);s)}{\xi_{\la}(t+\ep;s)-\xi_{\la}(t;s)}\times\frac{\xi_{\la}(t+\ep;s)-\xi_{\la}(t;s)}{\ep}\\
&\quad\quad=\frac{u(t,\xi_{\la}(t+\ep;s);s)-u(t+\ep,\xi_{\la}(t+\ep;s);s)}{\ep}
\end{split}
\end{equation*}
Since $\xi_{\la}(t;s)$ is continuous in $t$, we have
\begin{equation*}
\lim_{\ep\ra0}\frac{u(t,\xi_{\la}(t+\ep;s);s)-u(t,\xi_{\la}(t;s);s)}{\xi_{\la}(t+\ep;s)-\xi_{\la}(t;s)}=u_{x}(t,\xi_{\la}(t;s);s)<0
\end{equation*}
by Lemma \ref{lem-basic-prop}. In particular, $\frac{u(t,\xi_{\la}(t+\ep;s);s)-u(t,\xi_{\la}(t;s);s)}{\xi_{\la}(t+\ep;s)-\xi_{\la}(t;s)}\neq0$ for all small $\ep$. Thus,
\begin{equation*}
\begin{split}
&\frac{\xi_{\la}(t+\ep;s)-\xi_{\la}(t;s)}{\ep}\\
&\quad\quad=\frac{\xi_{\la}(t+\ep;s)-\xi_{\la}(t;s)}{u(t,\xi_{\la}(t+\ep;s);s)-u(t,\xi_{\la}(t;s);s)}\times\frac{u(t,\xi_{\la}(t+\ep;s);s)-u(t+\ep,\xi_{\la}(t+\ep;s);s)}{\ep}
\end{split}
\end{equation*}
Passing to the limit $\ep\ra0$ in the above equality, we conclude from the limit
\begin{equation*}
\lim_{\ep\ra0}\frac{u(t,\xi_{\la}(t+\ep;s);s)-u(t+\ep,\xi_{\la}(t+\ep;s);s)}{\ep}=-u_{t}(t,\xi_{\la}(t;s);s)
\end{equation*}
for $t>s$
that $\frac{d\xi_{\la}(t;s)}{dt}=\lim_{\ep\ra0}\frac{\xi_{\la}(t+\ep;s)-\xi_{\la}(t;s)}{\ep}$ exists and
\begin{equation*}
\frac{d\xi_{\la}(t;s)}{dt}=-\frac{u_{t}(t,\xi_{\la}(t;s);s)}{u_{x}(t,\xi_{\la}(t;s);s)}
\end{equation*}
for $t>s$,
which also implies the continuity of $\frac{d\xi_{\la}(t;s)}{dt}$ in $t$ for $t>s$. Hence, $\xi_{\la}(t;s)$ is continuously differentiable in
$t$ for $t>s$.
\end{proof}

We remark that due to the time-dependence of the nonlinear term $f(t,u)$, the time derivative $u_{t}(t,\xi_{\la}(t;s);s)$ does not have a fixed sign in general, and hence, $\frac{d\xi_{\la}(t;s)}{dt}$ does not have a fixed sign, which means $\xi_{\la}(t;s)$ oscillates and it is an unpleasant fact and does cause a lot of troubles (we point out that in the space heterogeneous case, the interface always propagates in one direction due to the time monotonicity, see \cite{MeRoSi10,NoRy09}). But, the estimate \eqref{estimate-lower-bound} forces $\xi_{\la}(t;s)$ to approach $+\infty$ as time $t$ elapses. However, the estimate \eqref{estimate-lower-bound} does not tell much information about how does $\xi_{\la}(t;s)$ approach $+\infty$. Later, in Lemma \ref{lem-rightward-prop-above-temp} and Lemma \ref{lem-propogation}, we characterize the rightward propagation of $\xi_{\la}(t;s)$, which plays the crucial role in deriving the boundedness of interface width and the exponential decay of the transition front ahead of the interface. We also note that $u_{t}(t,\xi_{\la}(t;s);s)$ is uniformly bounded in  $t\ge s+\delta_0$ for any $\delta_0>0$, but temporarily we are not sure if $u_{x}(t,\xi_{\la}(t;s);s)$ is uniformly away from $0$. But it is the case, see Theorem \ref{thm-non-flat}. Hence, $\frac{d\xi_{\la}(t;s)}{dt}$ is uniformly bounded in  $t\ge s+\delta_0$ for any $\delta_0>0$, that is, the interfaces cannot propagate faster than certain speed.


\section{Bounded Interface Width}\label{sec-bd-interface-width}

In this section, we show that the distance between their interface locations of  any two interfaces remains bounded as time elapses. Throughout this section, we consider \eqref{main-eqn} and assume $\rm(H1)$ and $\rm(H2)$. The main result of this section is given by

\begin{thm}\label{thm-bd-interface-width}
For any $\la_{1},\la_{2}\in(0,1)$, there exists $C=C(\la_{1},\la_{2})>0$ such that
\begin{equation*}
|\xi_{\la_{1}}(t;s)-\xi_{\la_{2}}(t;s)|\leq C(\la_{1},\la_{2})
\end{equation*}
for all $s<0$, $t\geq s$.
\end{thm}

To prove the above theorem, we first prove some lemmas and propositions. First of all, we  characterize the rightward propagation of interfaces above the ignition temperature. Let $f_{B}$ be a continuously differentiable function satisfying
\begin{equation}\label{bistable-nonlinearity}
\begin{split}
&f_{B}(0)=0,\,\,f_{B}(u)<0\,\,\text{for}\,\,u\in(0,\theta),\\
&f_{B}(u)=f_{\inf}(u)\,\,\text{for}\,\,u\in[\theta,1]\,\,\text{and}\,\,\int_{0}^{1}f_{B}(u)du>0.
\end{split}
\end{equation}
Since $f_{\inf}(u)>0$ for $u\in(\theta,1)$, such an $f_{B}$ exists. Clearly, $f_{B}$ is of standard bistable type and $f_{B}(u)\leq f(t,u)$ for all $u\in[0,1]$ and $t\in\R$. Hence, there exist (see e.g.\cite{ArWe75,ArWe78,FiMc77}) a unique constant $c_{B}>0$ and a wave profile $\phi_{B}$ satisfying $(\phi_{B})_{x}<0$, $\phi_{B}(-\infty)=1$ and $\phi_{B}(\infty)=0$ such that $\phi_{B}(x-c_{B}t)$ and its translations are traveling wave solutions of
\begin{equation}\label{eqn-bi-1}
u_{t}=u_{xx}+f_{B}(u).
\end{equation}

\begin{lem}\label{lem-rightward-prop-above-temp}
Let $\la\in(\theta,1)$. For any $\ep>0$, there is $t_{\ep,\la}>0$ such that
\begin{equation*}
\xi_{\la}(t;s)-\xi_{\la}(t_{0};s)\geq (c_{B}-\ep)(t-t_{0}-t_{\ep,\la})
\end{equation*}
for $s<0$, $t\geq t_{0}\geq s$.
\end{lem}
\begin{proof}
Let us fix a $\la\in(\theta,1)$. We first define
\begin{equation*}
\begin{split}
\psi_{*}(x)=\left\{\begin{aligned}
\la,&\quad x\leq0,\\
\max\{-C^{*}x+\la,0\},&\quad x\geq0,
\end{aligned} \right.
\end{split}
\end{equation*}
where $C^{*}>0$ is such that $\inf_{s<0,t\geq s}\inf_{x\in\R}u_{x}(t,x;s)\geq-C^{*}$. Such an $C^{*}$ exists by a priori estimates
for parabolic equations. Clearly, for any $s<0$ and $t_{0}\geq s$, we have
\begin{equation}\label{special-function-1}
\psi_{*}(x)\leq u(t_{0},x+\xi_{\la}(t_{0};s);s),\quad x\in\R.
\end{equation}

Next, for $t_{0}\geq s$, let $u_{B}(t,x;t_{0})$, $t\geq t_{0}$ be the solution of \eqref{eqn-bi-1} with initial data $u_{B}(t_{0},x;t_{0})=\psi_{*}(x)(\leq u(t_{0},x+\xi_{\la}(t_{0};s);s)$ by \eqref{special-function-1}). Thus, time homogeneity and comparison principle ensure
\begin{equation*}
u_{B}(t-t_{0},x;0)=u_{B}(t,x;t_{0})\leq u(t,x+\xi_{\la}(t_{0};s);s),\quad x\in\R,\,\,t\geq t_{0}.
\end{equation*}
By the stability of traveling wave solutions of \eqref{eqn-bi-1} (see \cite[Theorem 3.1]{FiMc77}) and the conditions satisfied by $\psi_{*}$, there exist $z_{0}=z_{0}(\la)\in\R$, $K=K(\la)>0$ and $\om=\om(\la)>0$ such that
\begin{equation*}
\sup_{x\in\R}|u_{B}(t-t_{0},x;0)-\phi_{B}(x-c_{B}(t-t_{0})-z_{0})|\leq Ke^{-\om(t-t_{0})}.\quad t\geq t_{0}.
\end{equation*}
In particular, for $t\geq t_{0}$ and $x\in\R$
\begin{equation}\label{estimate-l-bd}
u(t,x+\xi_{\la}(t_{0};s);s)\geq u_{B}(t-t_{0},x;0)\geq\phi_{B}(x-c_{B}(t-t_{0})-z_{0})-Ke^{-\om(t-t_{0})}.
\end{equation}
Let $T_{0}=T_{0}(\la)>0$ be such that $Ke^{-\om T_{0}}=\frac{1-\la}{2}$ and denote by $\xi_{B}(\frac{1+\la}{2})$ the unique point such that $\phi_{B}(\xi_{B}(\frac{1+\la}{2}))=\frac{1+\la}{2}$. Setting $x=c_{B}(t-t_{0})+z_{0}+\xi_{B}(\frac{1+\la}{2})$ in \eqref{estimate-l-bd}, we find for any $t\geq t_{0}+T_{0}$
\begin{equation*}
u(t,c_{B}(t-t_{0})+z_{0}+\xi_{B}(\frac{1+\la}{2})+\xi_{\la}(t_{0};s);s)\geq\phi_{B}(\xi_{B}(\frac{1+\la}{2}))-Ke^{-\om T_{0}}=\la.
\end{equation*}
Monotonicity then yields
\begin{equation}\label{propagation-result-1}
\xi_{\la}(t;s)-\xi_{\la}(t_{0};s)\geq c_{B}(t-t_{0})+z_{0}+\xi_{B}(\frac{1+\la}{2}), \quad t\geq t_{0}+T_{0}.
\end{equation}

Finally, we consider $\xi_{\la}(t;s)$ for $t\in[t_{0},t_{0}+T_{0}]$. To do so, let $\vp(x)=\max\{\hat{\vp}(x),0\}$ for $x\in\R$, where $\hat{\vp}$ is the unique solution of the following problem
\begin{equation*}
-\hat{\vp}_{xx}=f_{\inf}(\hat{\vp}),\quad \hat{\vp}(0)=\la,\quad \hat{\vp}_{x}(0)=0.
\end{equation*}
The function $\hat{\vp}$ satisfies the following properties:
\begin{itemize}
\item it is even and strictly decreasing for $x\geq0$;
\item it is strictly concave down for $x\in(-z_{1},z_{1})$, where $z_{1}>0$ is such that $\hat{\vp}(z_{1})=\theta$;
\item it is linear for $x\geq z_{1}$ with a negative slope.
\end{itemize}
Then, we can easily find a shift $z_{*}<0$ such that $\vp(x-z_{*})\leq\psi_{*}(x)$ for $x\in\R$. Denote by $u_{I}(t,x;t_{0})$ the solution of $u_{t}=u_{xx}+f_{\inf}(u)$ with $u_{I}(t_{0},x;t_{0})=\vp(x-z_{*})$. Since $-\vp_{xx}\leq f_{\inf}(\vp)$, we obtain from the maximum principle that $u_{I}(t,x;t_{0})\geq u_{I}(t_{0},x;t_{0})=\vp(x-z_{*})$ for all $t>t_{0}$. In particular, $u_{I}(t,z_{*};t_{0})\geq\la$ for all $t\geq t_{0}$.

Since $\vp(x-z_{*})\leq\psi_{*}(x)\leq u(t_{0},x+\xi_{\la}(t_{0};s);s)$, comparison principle implies that $u_{I}(t,x;t_{0})\leq u(t,x+\xi_{\la}(t_{0};s);s)$ for $t\geq t_{0}$. Setting $x=z_{*}$, we in particular have $u(t,z_{*}+\xi_{\la}(t_{0};s);s)\geq u_{I}(t,z_{*};t_{0})\geq\la$ for $t\geq t_{0}$. Monotonicity then yields
\begin{equation}\label{propagation-result-2}
\xi_{\la}(t;s)\geq z_{*}+\xi_{\la}(t_{0};s),\quad t\geq t_{0}.
\end{equation}
The result then follows from \eqref{propagation-result-1} and \eqref{propagation-result-2}.
\end{proof}

As seen in the proof of Lemma \ref{lem-rightward-prop-above-temp}, the bistable traveling waves $\phi_{B}(x-c_{B}t)$ push the approximation solutions $u(t,x;s)$ move rightward in some average sense. This property can also be derived if we use ignition traveling waves of $u_{t}=u_{xx}+f_{\inf}(u)$. The reason for using bistable traveling waves is that bistable traveling waves attract a larger class of initial data than ignition traveling waves do, and therefore, it is more flexible and convenient to use bistable traveling waves.

Next, for $\ka>0$, set $c^*_{\ka}=2\sqrt{\ka}$ and $\la_{\ka}=\sqrt{\ka}$. Clearly, $\la_{\ka}^{2}-c^*_{\ka}\la_{\ka}+\ka=0$, and hence, $e^{-\la_{\ka}x}$ is a solution of $\psi''+c^*_{\ka}\psi'+\ka\psi=0$. It is well-known that $c^*_{\ka}=\min_{\la>0}\frac{\ka+\la^{2}}{\la}=\frac{\ka+\la_{\ka}^{2}}{\la_{\ka}}$ is the minimal speed of a KPP traveling wave (see e.g. \cite{KPP37}).

For $\ka>0$, $s<0$ and $t\geq s$, define
\begin{equation}\label{new-interface}
\xi(t;s)=\inf\Big\{y\in\R\Big|u(t,x;s)\leq e^{-\la_{\ka}(x-y)},\quad x\in\R\Big\}.
\end{equation}
Due to the second estimate in \eqref{aprior-estimate-two-side1}, $\xi(t;s)$ is well-defined if $\la_{\ka}\leq c_{\inf}$, that is, $\ka\in(0,c_{\inf}^{2}]$. Here, we use the $\ka$-independent notation for $\xi(t;s)$, but this should not cause any trouble, since later in Lemma \ref{lem-bd-width}, we only need one small $\ka$. The following result controls the rightward propagation of $\xi(t;s)$.

\begin{lem}\label{lem-new-interface-bd}
Let $\ka\in(0,c_{\inf}^{2}]$. Set $\ka_{0}=\sup_{u\in(0,1)}\frac{f_{\sup}(u)}{u}$ and $c_{\ka_{0}}=\frac{\ka_{0}}{\la_{\ka}}+\la_{\ka}$. Then,
\begin{equation*}
\xi(t;s)-\xi(t_{0};s)\leq c_{\ka_{0}}(t-t_{0})
\end{equation*}
for all $s<0$, $t\geq t_{0}\geq s$.
\end{lem}
\begin{proof}
For $s<0$, $t\geq t_{0}\geq s$, define
\begin{equation*}
v(t,x;t_{0})=e^{-\la_{\ka}(x-\xi(t_{0};s)-c_{\ka_{0}}(t-t_{0}))}.
\end{equation*}
Since $c_{\ka_{0}}=\frac{\ka_{0}}{\la_{\ka}}+\la_{\ka}$, i.e., $\la_{\ka}^{2}-c_{\ka_{0}}\la_{\ka}+\ka_{0}=0$, we readily check that $v_{t}=v_{xx}+\ka_{0}v$. By the definition of $\ka_{0}$, we have $\ka_{0}v\geq f_{\sup}(v)$ for all $v\geq0$. It then follows from $v(t_{0},x;t_{0})=e^{-\la_{\ka}(x-\xi(t_{0};s))}\geq u(t_{0},x;s)$ by \eqref{new-interface} and the comparison principle that $v(t,x;t_{0})\geq u(t,x;s)$ for $t\geq t_{0}$, which leads to the result.
\end{proof}

Note the definition of $\xi(t;s)$ in \eqref{new-interface} and Lemma \ref{lem-new-interface-bd} does not guarantee any continuity of $\xi(t;s)$ in $t$. But, if we know $\xi(t;s)$ is increasing from $\xi(t_{0};s)$ for $t>t_{0}$, then it is controlled continuously by Lemma \ref{lem-new-interface-bd}. This observation is important in the next technical lemma, which is crucial in proving Theorem \ref{thm-bd-interface-width}.

\begin{lem}\label{lem-bd-width}
There exists $\la_{*}\in(\theta,1)$ such that for any $\la\in(\theta,\la_{*}]$, there is $C=C(\la)>0$ such that
\begin{equation*}
|\xi_{\la}(t;s)-\xi(t;s)|\leq C
\end{equation*}
for all $s<0$, $t\geq s$.
\end{lem}
\begin{proof}
We follow \cite[Lemma 2.5]{Zl13}.
Recall that for given $\ka>0$, $c^*_{\ka}=2\sqrt{\ka}$, and that  $c_{B}>0$ is the unique speed of traveling wave solutions of \eqref{eqn-bi-1}. We fix some $\ka\in(0,c_{\inf}^{2}]$ such that $c^*_{\ka}<c_{B}$, and set $\ep=\frac{c_{B}-c^*_{\ka}}{2}$ in Lemma \ref{lem-rightward-prop-above-temp}.
Let
\begin{equation*}
\la_{*}=\min\big\{u>0\big|f_{\sup}(u)=\ka u\big\}.
\end{equation*}
As a consequence, we have $f(t,u)\leq f_{\sup}(u)\leq \ka u$ for all $u\in[0,\la_{*}]$.

Fix an $\la\in(\theta,\la_{*}]$. Let $C_{0}=\max\{\xi(s;s)-\xi_{\la}(s;s),1\}$ (note $C_{0}$ is independent of $s$) and
$C_{1}=C_{0}+c_{B}t_{\ep,\la}$, where $t_{\ep,\la}$ is as in Lemma \ref{lem-rightward-prop-above-temp}. Notice the estimate $\xi_{\la}(t;s)-\xi(t;s)\leq C$ for some large $C>0$ is trivial. We show $\xi(t;s)-\xi_{\la}(t;s)\leq C_1$. Suppose this is not the case, then we can find some $t_{1}\geq s_{1}$ such $\xi(t_{1};s_{1})-\xi_{\la}(t_{1};s_{1})>C_1$. Let
\begin{equation*}
t_{0}=\sup\big\{t\in[s_{1},t_{1}]\big|\xi(t;s_{1})-\xi_{\la}(t;s_{1})\leq C_{0}\big\}.
\end{equation*}

We claim $\xi(t_{0};s_{1})-\xi_{\la}(t_{0};s_{1})\leq C_{0}$. It is trivial if there are only finitely many $t\in[s_{1},t_{1}]$ such that $\xi(t;s_{1})-\xi_{\la}(t;s_{1})\leq C_{0}$. So we assume there are infinitely many such $t$ and the claim is false. Then, there exists a sequence $\{\tilde{t}_{n}\}_{n\in\N}\subset[s_{1},t_{0})$ such that $\xi(\tilde{t}_{n};s_{1})-\xi_{\la}(\tilde{t}_{n};s_{1})\leq C_{0}$ for $n\in\N$ and $\tilde{t}_{n}\ra t_{0}$ as $n\ra\infty$. Moreover, $\xi(t_{0};s_{1})-\xi_{\la}(t_{0};s_{1})=\tilde C_{1}>C_{0}$. It then follows that for all $n\in\N$
\begin{equation*}
\xi(\tilde{t}_{n};s_{1})-\xi_{\la}(\tilde{t}_{n};s_{1})\leq C_{0}=C_{0}-\tilde C_{1}+\xi(t_{0};s_{1})-\xi_{\la}(t_{0};s_{1}),
\end{equation*}
that is,
\begin{equation*}
\tilde C_{1}-C_{0}+\xi_{\la}(t_{0};s_{1})-\xi_{\la}(\tilde{t}_{n};s_{1})\leq \xi(t_{0};s_{1})-\xi(\tilde{t}_{n};s_{1})\leq c_{\ka_{0}}(t_{0}-\tilde{t}_{n}),
\end{equation*}
where the second inequality is due to Lemma \ref{lem-new-interface-bd}.
Passing $n\ra\infty$, we conclude from the continuity of $\xi_{\la}(t;s_{1})$ in $t$ (see Lemma \ref{lem-negativity}) that $\tilde C_{1}-C_{0}\leq0$. It is a contradiction. Hence, the claim is true, that is, $\xi(t_{0};s_{1})-\xi_{\la}(t_{0};s_{1})\leq C_{0}$. It follows that $t_{0}<t_{1}$.

Instead of $\xi(t_{0};s_{1})-\xi_{\la}(t_{0};s_{1})\leq C_{0}$, there must hold
\begin{equation}\label{equality-technical}
\xi(t_{0};s_{1})-\xi_{\la}(t_{0};s_{1})=C_{0}.
\end{equation}
Suppose \eqref{equality-technical} is not true, then we can find some $\ep_{0}>0$ such that $\xi(t_{0};s_{1})-\xi_{\la}(t_{0};s_{1})=C_{0}-\ep_{0}$. Since $\xi(t;s_{1})-\xi_{\la}(t;s_{1})>C_{0}$ for $t\in(t_{0},t_{1}]$ by the definition of $t_{0}$, we deduce from Lemma \ref{lem-new-interface-bd} that for $t\in(t_{0},t_{1}]$
\begin{equation*}
\begin{split}
C_{0}<\xi(t;s_{1})-\xi_{\la}(t;s_{1})&\leq\xi(t_{0};s_{1})+c_{\ka_{0}}(t-t_{0})-\xi_{\la}(t_{0};s_{1})+\xi_{\la}(t_{0};s_{1})-\xi_{\la}(t;s_{1})\\
&=C_{0}-\ep_{0}+c_{\ka_{0}}(t-t_{0})+\xi_{\la}(t_{0};s_{1})-\xi_{\la}(t;s_{1}).
\end{split}
\end{equation*}
Since $\xi_{\la}(t;s_{1})$ is continuous in $t$, we fix some $t>t_{0}$ but close to $t_{0}$ such that $|c_{\ka_{0}}(t-t_{0})+\xi_{\la}(t_{0};s_{1})-\xi_{\la}(t;s_{1})|\leq\frac{\ep_{0}}{2}$, which then leads to $C_{0}<C_{0}-\frac{\ep_{0}}{2}$. It is a contradiction. Hence, \eqref{equality-technical} holds.

Next, we look at the time interval $[t_{0},t_{1}]$ and set $\tilde{\xi}(t;s_{1})=\xi(t_{0};s_{1})+c^*_{\ka}(t-t_{0})$ for $t\in[t_{0},t_{1}]$. Note both $\xi_{\la}(t;s_{1})$ and $\tilde{\xi}(t;s_{1})$ are continuous on $[t_{0},t_{1}]$ and $\xi_{\la}(t_{0};s_{1})<\tilde{\xi}(t_{0};s_{1})$ by \eqref{equality-technical}. We claim that $\xi_{\la}(t;s_{1})<\tilde{\xi}(t;s_{1})$ for all $t\in[t_{0},t_{1}]$. Suppose this is not the case and let $t_{2}=\min\{t\in[t_{0},t_{1}]|\xi_{\la}(t;s_{1})=\tilde{\xi}(t;s_{1})\}$. Clearly, $t_{2}\in(t_{0},t_{1}]$. Define
\begin{equation*}
v(t,x;t_{0})=e^{-\la_{\ka}(x-\tilde{\xi}(t;s_{1}))},\quad x\in\R,\,\, t\in[t_{0},t_{2}].
\end{equation*}
Since $\la_{\ka}^{2}-c^*_{\ka}\la_{\ka}+\ka=0$, we easily check $v_{t}=v_{xx}+\ka v$. Now, consider the parabolic domain
\begin{equation*}
D=\big\{(t,x)\in[t_{0},t_{2}]\times\R\big|x\geq\tilde{\xi}(t;s_{1})\big\}.
\end{equation*}
We see that for $(t,x)\in D$, $x\geq\tilde{\xi}(t;s_{1})\geq\xi_{\la}(t;s_{1})$, which leads to $u=u(t,x;s_{1})\leq\la\in(0,\la_{*}]$ by monotonicity, and then, $f(t,u)\leq \ka u$ as noted in the beginning of the proof. Also, at the initial moment $t_{0}$, we have $u(t_{0},x;s_{1})\leq e^{-\la_{\ka}(x-\xi(t_{0};s_{1}))}=v(t_{0},x;t_{0})$, and at the boundary point $x=\tilde{\xi}(t;s_{1})$, we trivially have $u(t,\tilde{\xi}(t;s_{1});s_{1})<1=v(t,\tilde{\xi}(t;s_{1});t_{0})$. Thus, comparison principle yields $u(t,x;s_{1})\leq v(t,x;t_{0})$ on $D$, which leads to
\begin{equation*}
u(t,x;s_{1})\leq v(t,x;t_{0})=e^{-\la_{\ka}(x-\tilde{\xi}(t;s_{1}))},\quad x\in\R,\,\, t\in[t_{0},t_{2}].
\end{equation*}
It follows that $\xi(t;s_{1})\leq\tilde{\xi}(t;s_{1})$ for $t\in[t_{0},t_{2}]$ by definition in \eqref{new-interface}. In particular, $\xi(t_{2};s_{1})\leq\tilde{\xi}(t_{2};s_{1})=\xi_{\la}(t_{2};s_{1})$. Since $t_{2}\in(t_{0},t_{1}]$, we have $\xi(t_{2};s_{1})-\xi_{\la}(t_{2};s_{1})>C_{0}$ by the definition of $t_{0}$. It is a contradiction. Thus, the claim follows, that is, $\xi_{\la}(t;s_{1})<\tilde{\xi}(t;s_{1})$ for all $t\in[t_{0},t_{1}]$, and repeating the above arguments, we see
\begin{equation}\label{propagation-estimate-refined}
\xi(t;s_{1})\leq\tilde{\xi}(t;s_{1})=\xi(t_{0};s_{1})+c^*_{\ka}(t-t_{0}),\quad t\in[t_{0},t_{1}].
\end{equation}

It follows from \eqref{propagation-estimate-refined} and  Lemma \ref{lem-rightward-prop-above-temp} that for any $t\in[t_{0},t_{1}]$
\begin{equation*}
\begin{split}
\xi(t;s_{1})-\xi_{\la}(t;s_{1})&\leq\xi(t_{0};s_{1})+c^*_{\ka}(t-t_{0})-[\xi_{\la}(t_{0};s_{1})+(c_{B}-\ep)(t-t_{0}-t_{\ep,\la})]\\
&\leq C_{0}+(c_{B}-\ep)t_{\ep,\la}-(c_{B}-c^*_{\ka}-\ep)(t-t_{0})\\
&\leq C_{0}+c_{B}t_{\ep,\la}\\
&=C_1.
\end{split}
\end{equation*}
Thus,  we in particular have $\xi(t_{1};s_{1})-\xi_{\la}(t_{1};s_{1})\leq C_{1}$, which is a contradiction. Consequently, $\xi(t;s)-\xi_{\la}(t;s)\leq C_{1}$ for all $s<0$, $t\geq s$. This completes the proof.
\end{proof}

The following proposition is in fact Theorem \ref{thm-bd-interface-width} restricted to the case $\la_{1},\la_{2}\in(0,\la_{*}]$.

\begin{prop}\label{prop-bd-interface-width}
For any $\la_{1},\la_{2}\in(0,\la_{*}]$ there exists $C=C(\la_{1},\la_{2})>0$ such that
\begin{equation*}
|\xi_{\la_{1}}(t;s)-\xi_{\la_{2}}(t;s)|\leq C
\end{equation*}
for all $s<0$ and $t\geq s$, where $\la_{*}\in(\theta,1)$ is as in Lemma \ref{lem-bd-width}.
\end{prop}

\begin{proof}
Fix some $\ka\in(0,c_{\inf}^{2}]$ such that $c^*_{\ka}<c_{B}$ as in the proof of Lemma \ref{lem-bd-width}. Let $\la_{*}\in(\theta,1)$ be as in Lemma \ref{lem-bd-width} and $\la_{1},\la_{2}\in(0,\la_{*}]$. We may assume, without loss of generality, that $\la_{1}<\la_{2}$. Thus, $\xi_{\la_{1}}(t;s)\geq\xi_{\la_{2}}(t;s)$, and
\begin{equation*}
\xi_{\la_{1}}(t;s)-\xi_{\la_{2}}(t;s)\leq\eta_{\la_{1}}(t;s)-\xi_{\la_{*}}(t;s),
\end{equation*}
where $\eta_{\la_{1}}(t;s)$ is the unique point such that $e^{-\la_{\ka}(\eta_{\la_{1}}(t;s)-\xi(t;s))}=\la_{1}$. Since $\eta_{\la_{1}}(t;s)-\xi(t;s)\equiv \hat{C}(\la_{1})$ for some $\hat{C}(\la_{1})>0$, we deduce from Lemma \ref{lem-bd-width} that $\xi_{\la_{1}}(t;s)-\xi_{\la_{2}}(t;s)\leq \hat{C}(\la_{1})+C(\la_{*})$.
\end{proof}

Note that in the presence of Proposition \ref{prop-bd-interface-width}, to finish the proof of Theorem \ref{thm-bd-interface-width}, we only need to bound $\xi_{\theta}(t;s)-\xi_{\la}(t;s)$ for all $\la\in(\theta,1)$ close to $1$. To do so, we need to study the propagation of $\xi_{\theta}(t;s)$.

Let $u^{o}(t;t_{0},a)$ be the solution of the ODE $u_{t}=f(t,u)$ with initial data $u^{o}(t_{0};t_{0},a)=a$. Let $\de\in(0,1-\theta)$. For $t_{0}\in\R$, $t\geq t_{0}$ and $x\in\R$, define
\begin{equation*}
\begin{split}
\om_{+}(t,x;t_{0})&=(\theta-\de)\big[1-\phi(x-x_{s}-C(t-t_{0}))\big]+u^{o}(t;t_{0},1+\de)\phi(x-x_{s}-C(t-t_{0})),\\
\om_{-}(t,x;t_{0})&=-\de\big[1-\phi(x+x_{s}+C(t-t_{0}))\big]+u^{o}(t;t_{0},\theta+\de)\phi(x+x_{s}+C(t-t_{0})),
\end{split}
\end{equation*}
where $C>0$ is some constant. Note that $u^{o}(t;t_{0},1+\de)$ and $u^{o}(t;t_{0},\theta+\de)$ are decreasing and increasing in $t$, respectively, and
$$
\lim_{t\ra\infty}u^{o}(t;t_{0},1+\de)=1=\lim_{t\ra\infty}u^{o}(t;t_{0},\theta+\de).
$$

\begin{lem}\label{lem-sub-super-sol}
For sufficiently large $C>0$, $\om_{+}(t,x;t_{0})$ and $\om_{-}(t,x;t_{0})$ are sup-solution and sub-solution of \eqref{main-eqn}, respectively.
\end{lem}
\begin{proof}
We only prove that $\om_{+}(t,x;t_{0})$ is a super-solution for sufficiently large $C$; $\om_{-}(t,x;t_{0})$ being a sub-solution for sufficiently large $C$ can be proven similarly. We compute
\begin{equation*}
\begin{split}
&(\om_{+})_{t}-(\om_{+})_{xx}-f(t,\om_{+})\\
&\quad\quad=(\theta-\de-u^{o}(t;t_{0},1+\de))\big[C\phi'(x-x_{s}-C(t-t_{0}))+\phi''(x-x_{s}-C(t-t_{0}))\big]\\
&\quad\quad\quad+f(t,u^{o}(t;t_{0},1+\de))\phi(x-x_{s}-C(t-t_{0}))-f(t,\om_{+})\\
&\quad\quad=(C-c_{\inf})(\theta-\de-u^{o}(t;t_{0},1+\de))\phi'(x-x_{s}-C(t-t_{0}))\\
&\quad\quad\quad+(u^{o}(t;t_{0},1+\de)-\theta+\de)f_{\inf}(\phi(x-x_{s}-C(t-t_{0})))\\
&\quad\quad\quad+f(t,u^{o}(t;t_{0},1+\de))\phi(x-x_{s}-C(t-t_{0}))-f(t,\om_{+})\\
&\quad\quad\geq(C-c_{\inf})(\theta-\de-u^{o}(t;t_{0},1+\de))\phi'(x-x_{s}-C(t-t_{0}))\\
&\quad\quad\quad+f(t,u^{o}(t;t_{0},1+\de))\phi(x-x_{s}-C(t-t_{0}))-f(t,\om_{+}),\\
\end{split}
\end{equation*}
where we used the equation in \eqref{traveling-homo} in the second equality.

There are two cases. If $\om_{+}\leq\theta$, then $f(t,\om_{+})=0$ and $u^{o}(t;t_{0},1+\de)\phi(x-x_{s}-C(t-t_{0}))\leq\theta$, which forces $\phi(x-x_{s}-C(t-t_{0}))\leq\theta$ and hence, $x-x_{s}-C(t-t_{0})\geq0$ by monotonicity. We then conclude from \eqref{explicit-sol} or the way $(\phi(z),\phi'(z))$ approaches $(0,0)$ as $z\ra\infty$ that $\phi(x-x_{s}-C(t-t_{0}))$ and $\phi'(x-x_{s}-C(t-t_{0}))$ are comparable, which leads to
\begin{equation}\label{sup-sol-propagation}
(\om_{+})_{t}-(\om_{+})_{xx}-f(t,\om_{+})\geq0
\end{equation}
for sufficiently large $C>0$.

If $\om_{+}>\theta$, then by Taylor expansion,
\begin{equation*}
\begin{split}
&f(t,u^{o}(t;t_{0},1+\de))\phi(x-x_{s}-C(t-t_{0}))-f(t,\om_{+})\\
&\quad\quad=[f(t,u^{o}(t;t_{0},1+\de))-f(t,\om_{+})\big]\phi(x-x_{s}-C(t-t_{0}))\\
&\quad\quad\quad+f(t,\om_{+})\big[1-\phi(x-x_{s}-C(t-t_{0}))\big]\\
&\quad\quad=f_{u}(t,u_{*})(u^{o}(t;t_{0},1+\de)-\theta+\de)\big[1-\phi(x-x_{s}-C(t-t_{0}))\big]\phi(x-x_{s}-C(t-t_{0}))\\
&\quad\quad\quad+f(t,\om_{+})\big[1-\phi(x-x_{s}-C(t-t_{0}))\big],
\end{split}
\end{equation*}
where $u_{*}\in[\om_{+},u^{o}(t;t_{0},1+\de)]$. Note that the condition $\om_{+}>\theta$ forces $x-x_{s}-C(t-t_{0})\leq x_{*}$ for some universal constant $x_{*}>0$. We then conclude from the way $(\phi(z),\phi'(z))$ approaches $(1,0)$ as $z\ra-\infty$ that $1-\phi(x-x_{s}-C(t-t_{0}))$ and $\phi'(x-x_{s}-C(t-t_{0}))$ are comparable, and hence, \eqref{sup-sol-propagation} holds as well for sufficiently large $C>0$.
\end{proof}

The next result concerns the propagation of $\xi_{\theta}(t;s)$.

\begin{prop}\label{prop-bd-propagation}
For any $T_{0}>0$, there exists $h_{0}=h_{0}(T_{0})>0$ such that
\begin{equation*}
|\xi_{\theta}(t+T_{0};s)-\xi_{\theta}(t;s)|\leq h_{0}
\end{equation*}
for all $s<0$, $t\geq s$. Moreover, $h_{0}(T_{0})$ is increasing in $T_{0}$.
\end{prop}

\begin{proof}
Let $\de_{*}=\la_{*}-\theta$, where $\la_{*}$ is as in  Lemma \ref{lem-bd-width}. First, since
\begin{equation*}
\begin{split}
\om_{+}(t_{0},x;t_{0})&=(\theta-\de_{*})\big[1-\phi(x-x_{s})\big]+(1+\de_{*})\phi(x-x_{s}),\\
\om_{-}(t_{0},x;t_{0})&=-\de_{*}\big[1-\phi(x+x_{s})\big]+(\theta+\de_{*})\phi(x+x_{s}),
\end{split}
\end{equation*}
we can find some $x_{*}<0$ such that $\om_{+}(t_{0},x_{*}+x_{s};t_{0})\geq1$ and $\om_{-}(t_{0},-x_{*}-x_{s};t_{0})\leq0$,
which yields
\begin{equation*}
\om_{-}(t_{0},x-\xi_{\theta+\de_{*}}(t_{0};s)-x_{*}-x_{s};t_{0})\leq u(t_{0},x;s)\leq\om_{+}(t_{0},x-\xi_{\theta-\de_{*}}(t_{0};s)+x_{*}+x_{s};t_{0})
\end{equation*}
for all $x\in\R$. It then follows from Lemma \ref{lem-sub-super-sol} and comparison principle that
\begin{equation}\label{estimate-lower-upper-bd}
\om_{-}(t,x-\xi_{\theta+\de_{*}}(t_{0};s)-x_{*}-x_{s};t_{0})\leq u(t,x;s)\leq\om_{+}(t,x-\xi_{\theta-\de_{*}}(t_{0};s)+x_{*}+x_{s};t_{0})
\end{equation}
for all $x\in\R$ and $t\geq t_{0}$.

We now fix any $T_{0}>0$. Since
\begin{equation*}
\begin{split}
\om_{+}(t_{0}+T_{0},x+CT_{0};t_{0})&=(\theta-\de_{*})\big[1-\phi(x-x_{s})\big]+u^{o}(t_{0}+T_{0};t_{0},1+\de_{*})\phi(x-x_{s}),\\
&\leq(\theta-\de_{*})\big[1-\phi(x-x_{s})\big]+(1+\de_{*})\phi(x-x_{s}),\\
\om_{-}(t_{0}+T_{0},x-CT_{0};t_{0})&=-\de_{*}\big[1-\phi(x+x_{s})\big]+u^{o}(t_{0}+T_{0};t_{0},\theta+\de_{*})\phi(x+x_{s}),\\
&\geq-\de_{*}\big[1-\phi(x+x_{s})\big]+(\theta+\de_{*})\phi(x+x_{s}),
\end{split}
\end{equation*}
we can find some $x_{**}>0$ such that
\begin{equation*}
\begin{split}
\om_{+}(t_{0}+T_{0},x_{**}+x_{s}+CT_{0};t_{0})&\leq\theta-\frac{\de_{*}}{2},\\
\om_{-}(t_{0}+T_{0},-x_{**}-x_{s}-CT_{0};t_{0})&\geq\theta+\frac{\de_{*}}{2}.
\end{split}
\end{equation*}
This together with \eqref{estimate-lower-upper-bd} gives
\begin{equation*}
\begin{split}
u(t_{0}+T_{0},\xi_{\theta-\de_{*}}(t_{0};s)-x_{*}+x_{**}+CT_{0};s)&\leq\theta-\frac{\de_{*}}{2},\\
u(t_{0}+T_{0},\xi_{\theta+\de_{*}}(t_{0};s)+x_{*}-x_{**}-CT_{0};s)&\geq\theta+\frac{\de_{*}}{2}.
\end{split}
\end{equation*}
By monotonicity, we find
\begin{equation}\label{estimate-propagation-aux}
\begin{split}
\xi_{\theta-\frac{\de_{*}}{2}}(t_{0}+T_{0};s)&\leq\xi_{\theta-\de_{*}}(t_{0};s)-x_{*}+x_{**}+CT_{0},\\
\xi_{\theta+\frac{\de_{*}}{2}}(t_{0}+T_{0};s)&\geq\xi_{\theta+\de_{*}}(t_{0};s)+x_{*}-x_{**}-CT_{0}.
\end{split}
\end{equation}

Finally, to finish the proof, we set
$$
h_{0}=h_{0}(T_{0})=-x_{*}+x_{**}+CT_{0}+C(\theta+\de_{*},\theta-\de_{*}),
$$
where $C(\theta+\de_{*},\theta-\de_{*})>0$ is as in Proposition \ref{prop-bd-interface-width}. Then,
\begin{equation*}
\begin{split}
&\xi_{\theta}(t_{0}+T_{0};s)-\xi_{\theta}(t_{0};s)\\
&\quad\quad\leq\xi_{\theta-\frac{\de_{*}}{2}}(t_{0}+T_{0};s)-\xi_{\theta+\de_{*}}(t_{0};s)\\
&\quad\quad=\xi_{\theta-\frac{\de_{*}}{2}}(t_{0}+T_{0};s)-\xi_{\theta-\de_{*}}(t_{0};s)+\xi_{\theta-\de_{*}}(t_{0};s)-\xi_{\theta+\de_{*}}(t_{0};s)\\
&\quad\quad\leq-x_{*}+x_{**}+CT_{0}+C(\theta+\de_{*},\theta-\de_{*})=h_{0},
\end{split}
\end{equation*}
where we used the first estimate in \eqref{estimate-propagation-aux} and Proposition \ref{prop-bd-interface-width}. Similarly, by the second estimate in \eqref{estimate-propagation-aux} and Proposition \ref{prop-bd-interface-width}, we deduce $\xi_{\theta}(t_{0}+T_{0};s)-\xi_{\theta}(t_{0};s)\geq-h_{0}$. This completes the proof.
\end{proof}

Finally, we prove Theorem \ref{thm-bd-interface-width}.

\begin{proof}[Proof of Theorem \ref{thm-bd-interface-width}]
Note that in the presence of Proposition \ref{prop-bd-interface-width}, we only need to bound $\xi_{\theta}(t;s)-\xi_{\la}(t;s)$ for all $\la\in(\theta,1)$ close to $1$. To do so, let $\de_{*}=\la_{*}-\theta$ be as in the proof of Proposition \ref{prop-bd-propagation}. Recall that $u^{o}(t;t_{0},\theta+\de_{*})$ is increasing in $t$ and $\lim_{t\ra\infty}u^{o}(t;t_{0},\theta+\de_{*})=1$. From which, we can find some $T_{-}>0$ and $x_{-}<0$ such that
\begin{equation*}
\om_{-}(t,x_{-}-x_{s}-C(t-t_{0});t_{0})\geq1-\de_{*},\quad t\geq t_{0}+T_{-}.
\end{equation*}
Using the first inequality in \eqref{estimate-lower-upper-bd}, we find
\begin{equation*}
u(t,\xi_{\theta+\de_{*}}(t_{0};s)+x_{*}+x_{-}-C(t-t_{0});s)\geq1-\de_{*},\quad t\geq t_{0}+T_{-}.
\end{equation*}
By monotonicity,
\begin{equation*}
\xi_{1-\de_{*}}(t;s)\geq\xi_{\theta+\de_{*}}(t_{0};s)+x_{*}+x_{-}-C(t-t_{0}),\quad t\geq t_{0}+T_{-}.
\end{equation*}
Setting $t=t_{0}+T_{-}$ in the above estimate, we find
\begin{equation*}
\xi_{1-\de_{*}}(t_{0}+T_{-};s)\geq\xi_{\theta+\de_{*}}(t_{0};s)+x_{*}+x_{-}-CT_{-}.
\end{equation*}
Since $\xi_{\theta}(t_{0}+T_{-};s)\leq\xi_{\theta}(t_{0};s)+h_{0}(T_{-})$ by Proposition \ref{prop-bd-propagation}, we find
\begin{equation*}
\begin{split}
&\xi_{\theta}(t_{0}+T_{-};s)-\xi_{1-\de_{*}}(t_{0}+T_{-};s)\\
&\quad\quad\leq\xi_{\theta}(t_{0};s)-\xi_{\theta+\de_{*}}(t_{0};s)+h_{0}(T_{-})-x_{*}-x_{-}+CT_{-}\\
&\quad\quad\leq\ep_{*}+h_{0}(T_{-})-x_{*}-x_{-}+CT_{-}
\end{split}
\end{equation*}
by Proposition \ref{prop-bd-interface-width}, where $\ep_{*}=C(\theta,\theta+\de_{*})$. Since $t_{0}\geq s$ is arbitrary, we arrive at
\begin{equation*}
\xi_{\theta}(t;s)-\xi_{1-\de_{*}}(t;s)\leq\ep_{*}+h_{0}(T_{-})-x_{*}-x_{-}+CT_{-},\quad t\geq s+T_{-}.
\end{equation*}

For the time interval $[s,s+T_{-}]$, we consider space-time homogeneous equations
\begin{equation}\label{eqn-space-time-homo}
\begin{split}
u_{t}=u_{xx}+f_{\inf}(u),\quad u_{t}=u_{xx}+f_{\sup}(u).
\end{split}
\end{equation}
Let $u_{\inf}(t,x;s)$ and $u_{\sup}(t,x;s)$ be solutions of the first and the second equation in \eqref{eqn-space-time-homo}, respectively, with $u_{\inf}(s,x;s)=\phi(x-x_{s})=u_{\sup}(s,x;s)$. By comparison principle and homogeneity, we find
\begin{equation*}
u_{\inf}(t-s,x;0)\leq u(t,x;s)\leq u_{\sup}(t-s,x;0),\quad x\in\R,\,\,t\geq s.
\end{equation*}
Denote by $\xi_{1-\de_{*}}^{\inf}(t-s)$ be the unique point such that $u_{\inf}(t-s,\xi_{1-\de_{*}}^{\inf}(t-s);0)=1-\de_{*}$ and by $\xi_{\theta}^{\sup}(t-s)$ be the unique point such that $u_{\sup}(t-s,\xi_{\theta}^{\sup}(t-s);0)=\theta$. Then, for $t\in[s,s+T_{-}]$ we have
\begin{equation*}
-\infty<\inf_{t\in[s,s+T_{-}]}\xi_{1-\de_{*}}^{\inf}(t-s)\leq\xi_{1-\de_{*}}(t;s)<\xi_{\theta}(t;s)\leq\sup_{t\in[s,s+T_{-}]}\xi_{\theta}^{\sup}(t-s)<\infty.
\end{equation*}
Setting
\begin{equation*}
\ep_{**}=\sup_{t\in[s,s+T_{-}]}\xi_{\theta}^{\sup}(t-s)-\inf_{t\in[s,s+T_{-}]}\xi_{1-\de_{*}}^{\inf}(t-s),
\end{equation*}
we find $\xi_{\theta}(t;s)-\xi_{1-\de_{*}}(t;s)\leq\ep_{**}$ for $t\in[s,s+T_{-}]$. Thus, setting
\begin{equation*}
\ep_{***}=\max\big\{\ep_{**},\ep_{*}+h_{0}(T_{-})-x_{*}-x_{-}+CT_{-}\big\},
\end{equation*}
we have
\begin{equation}
\label{width-bd-eq}
\xi_\theta(t;s)-\xi_{1-\delta_{*}}(t;s)\le\epsilon_{***},\quad s<0,\,\, t\geq s.
\end{equation}

The theorem then follows from Proposition \ref{prop-bd-interface-width} and \eqref{width-bd-eq}.
\end{proof}


\section{Uniform Steepness Estimate}\label{Non-Flat Estimate Near Ignition Temperature}

This section is devoted to the uniform steepness of $u(t,x;s)$ near $\xi_{\theta}(t;s)$. Through this section, we assume $\rm(H1)$ and $\rm(H2)$. The main result is the following

\begin{thm}\label{thm-non-flat}
There exist a constant $T_{D}>0$ and a continuous nonincreasing function $\al:[0,\infty)\ra(0,\infty)$ such that for any $M\geq0$ there holds
\begin{equation*}
u_{x}(t,x;s)\leq-\al(M),\quad x\in[\xi_{\theta}(t;s)-M,\xi_{\theta}(t;s)+M]
\end{equation*}
for all $s<0$, $t\geq s+T_{D}$. In particular, the following statements hold.

\begin{itemize}
\item[\rm(i)] For any $\la\in(0,1)$, there is $\al_{\la}>0$ such that
\begin{equation*}
u_{x}(t,\xi_{\la}(t;s);s)\leq-\al_{\la}
\end{equation*}
for all $s<0$ and $t\geq s+T_{D}$. Moreover, the function $\la\mapsto\al_{\la}:(0,1)\ra(0,\infty)$ is continuous and bounded.

\item[\rm(ii)] For any $\la\in(0,1)$, there exists $C_{\la}>0$ such that
\begin{equation*}
\sup_{s<0,t\geq s+T_{D}}\bigg|\frac{d\xi_{\la}(t;s)}{dt}\bigg|\leq C_{\la}.
\end{equation*}
\end{itemize}
\end{thm}

The notation $T_{D}$ stands for the time delay. We understand it as the time that the approximation solutions take to adjust their shapes. The proof of Theorem \ref{thm-non-flat} depends on the boundedness of interface width as in Theorem \ref{thm-bd-interface-width} and the propagation of the interface location $\xi_{\theta}(t;s)$ as in Proposition \ref{prop-bd-propagation}. To prove Theorem \ref{thm-non-flat}, we first prove a lemma.

\begin{lem}[\cite{Sh99-1}]\label{derivative-integral-estimate}
For any $h>0$, $t\geq t_{0}\geq s$, there holds
\begin{equation*}
u_{x}(t,x;s)\leq J(t-t_{0},|x-z|+h)\int_{z-h}^{z+h}u_{x}(t_{0},y;s)dy,
\end{equation*}
where
\begin{equation*}
J(t-t_{0},|x-z|+h)=e^{-\tilde{M}(t-t_{0})}\frac{1}{\sqrt{4\pi(t-t_{0})}}e^{-\frac{(|x-z|+h)^{2}}{4(t-t_{0})}}
\end{equation*}
for some $\tilde{M}>0$.
\begin{proof}
Let $\ep>0$. Set $v_{1}(t,x;s)=u(t,x+\ep;s)$ and $v_{2}(t,x;s)=u(t,x;s)$. By monotonicity, $v_{1}(t,x;s)<v_{2}(t,x;s)$. Clearly, $v(t,x;s)=v_{1}(t,x;s)-v_{2}(t,x;s)$ satisfies
\begin{equation*}
v_{t}=v_{xx}+f(t,v_{1})-f(t,v_{2}).
\end{equation*}
By (H1), there exists $\tilde{M}>0$ such that $f(t,v_{1})-f(t,v_{2})\leq-\tilde{M}(v_{1}-v_{2})$, and hence
\begin{equation*}
v_{t}\leq v_{xx}-\tilde{M}v.
\end{equation*}
By comparison principle, we obtain for $t\geq t_{0}\geq s$
\begin{equation*}
\begin{split}
&u(t,x+\ep;s)-u(t,x;s)\\
&\quad\quad=v(t,x;s)\\
&\quad\quad\leq e^{-\tilde{M}(t-t_{0})}\int_{\R}\frac{1}{\sqrt{4\pi(t-t_{0})}}e^{-\frac{(x-y)^{2}}{4(t-t_{0})}}[u(t_{0},y+\ep;s)-u(t_{0},y;s)]dy\\
&\quad\quad\leq e^{-\tilde{M}(t-t_{0})}\int_{z-h}^{z+h}\frac{1}{\sqrt{4\pi(t-t_{0})}}e^{-\frac{(x-y)^{2}}{4(t-t_{0})}}[u(t_{0},y+\ep;s)-u(t_{0},y;s)]dy\\
&\quad\quad\leq e^{-\tilde{M}(t-t_{0})}\frac{1}{\sqrt{4\pi(t-t_{0})}}e^{-\frac{(|x-z|+h)^{2}}{4(t-t_{0})}}\int_{z-h}^{z+h}[u(t_{0},y+\ep;s)-u(t_{0},y;s)]dy,
\end{split}
\end{equation*}
which leads to the result.
\end{proof}
\end{lem}

Observe that $J(t-t_{0},|x-z|+h)\ra0$ as $t-t_{0}\ra0$, that is, the estimate given in Lemma \ref{derivative-integral-estimate} is degenerate when $t$ approaches $t_{0}$. This is the technical reason why we introduce the time delay $T_{D}$ in the statement of Theorem \ref{thm-non-flat}.

We now prove Theorem \ref{thm-non-flat}.

\begin{proof}[Proof of Theorem \ref{thm-non-flat}]
Set $h_{\theta}=\max\{C(\theta,\frac{1+\theta}{2}),C(\theta,\frac{\theta}{2})\}$, where $C(\theta,\frac{1+\theta}{2})$ and $C(\theta,\frac{\theta}{2})$ are as in Theorem \ref{thm-bd-interface-width}. Theorem \ref{thm-bd-interface-width} then ensures that for all $t_{0}\geq s$
\begin{equation}\label{an-inequality-aux-1234}
\xi_{\theta}(t_{0};s)+h_{\theta}\geq\xi_{\frac{\theta}{2}}(t_{0};s),\quad \xi_{\theta}(t_{0};s)-h_{\theta}\leq\xi_{\frac{1+\theta}{2}}(t_{0};s).
\end{equation}

Now, for any $\tau\geq0$ and $t_{0}\geq s$, applying Lemma \ref{derivative-integral-estimate} with $z=\xi_{\theta}(t_{0};s)$ and $h=h_{\theta}$, we obtain that if $|x-\xi_{\theta}(t_{0};s)|\leq M$, then
\begin{equation}\label{an-estimate-aux-1234}
\begin{split}
u_x(\tau+t_{0},x;s)&\leq J(\tau, M+h_{\theta})\int_{\xi_{\theta}(t_{0};s)-h_{\theta}}^{\xi_{\theta}(t_{0};s)+h_{\theta}}u_{x}(t_{0},y;s)dy\\
&=J(\tau, M+h_{\theta})\big[u(t_{0},\xi_{\theta}(t_{0};s)+h_{\theta};s)-u(t_{0},\xi_{\theta}(t_{0};s)-h_{\theta};s)\big]\\
&\leq J(\tau, M+h_{\theta})\big[u(t_{0},\xi_{\frac{\theta}{2}}(t_{0};s);s)-u(t_{0},\xi_{\frac{1+\theta}{2}}(t_{0};s);s)\big]\\
&=-\frac{1}{2}J(\tau,M+h_{\theta}),
\end{split}
\end{equation}
where we used \eqref{an-inequality-aux-1234} and monotonicity of $u(t_{0},x;s)$ in $x$ in the second inequality.

Finally, fix some $T_{0}$, where $T_{0}$ is as in Proposition \ref{prop-bd-propagation}. Setting $\tau=T_{0}$ in \eqref{an-estimate-aux-1234}, we find that if $|x-\xi_{\theta}(t_{0}+T_{0};s)|\leq M$, then
\begin{equation*}
|x-\xi_{\theta}(t_{0};s)|\leq|x-\xi_{\theta}(t_{0}+T_{0};s)|+|\xi_{\theta}(t_{0}+T_{0};s)-\xi_{\theta}(t_{0};s)|\leq M+h_{0}(T_{0},\de_{*})
\end{equation*}
by Proposition \ref{prop-bd-propagation}, and hence,
$$
u_x(t_{0}+T_{0},x;s)\leq-\frac{1}{2}J(T_{0},M+h_{0}(T_{0},\de_{*})+h_{\theta}).
$$
This completes the proof of the main result. For the ``in particular" part, we argue as follows.

$\rm(i)$ It is a simple consequence of the just-proven result and  Theorem \ref{thm-bd-interface-width}.

$\rm(ii)$ It follows from Lemma \ref{lem-negativity}, the uniform boundedness of $u_{t}(t,\xi_{\theta}(t;s);s)$ in $t\ge s+\delta_0$ for any
$\delta_0>0$ and $\rm(i)$.
\end{proof}




\section{Uniform Decaying Estimates}\label{uniform-estimate}

 In this section, we investigate the uniform-in-time estimates of $u(t,x+\xi_\theta(t;s),s)$ for $x\le 0$ (referred to as  behind the interface) and $x\ge 0$ (referred to as ahead of the interface). Throughout this section, we assume $\rm(H1)$ and $\rm(H2)$.

\subsection{Uniform Decaying Estimates Behind Interface}

In this subsection, we control $u(t,x;s)$ behind the interface. The main results of this subsection are stated in the following theorem.

\begin{thm}\label{thm-uniform-estimate-hehind}
\begin{itemize}
\item[\rm(i)] There is a strictly decreasing function $v:(-\infty,0]\ra[\theta,1)$ satisfying $v(-\infty)=1$ and $v(0)=\theta$ such that
\begin{equation*}
u(t,x+\xi_{\theta}(t;s);s)\geq v(x),\quad x\leq0
\end{equation*}
for all $s<0$, $t\geq s+T_{D}$, where $T_{D}$ is given in Theorem \ref{thm-non-flat}.

\item[\rm(ii)] There exist $\la_{0}\in(\theta,1)$, $r>0$ and $\beta_{0}>0$ such that
\begin{equation*}
u(t,x+\xi_{\theta}(t;s);s)\geq1-(1-\la_{0})\Big[e^{-\be_{0}(t-s)}+e^{r(x+C(\theta,\la_{0}))}\Big],\quad x\leq-C(\theta,\la_{0})
\end{equation*}
for all $s<0$, $t\geq s+T_{D}$, where $C(\theta,\la_{0})$ is as in Theorem \ref{thm-bd-interface-width}.
\end{itemize}
\end{thm}

The first part of the theorem gives an uniform control of $u(t,x;s)$ behind the interface. The second part gives an exponential property of $1-u(t,x;s)$ behind the interface, which leads to the exponential decay behind the interface of the limiting function $1-u(t,x;s)$ as $s\ra-\infty$.

To  prove Theorem \ref{thm-uniform-estimate-hehind}, we first prove a  lemma giving the exponential property of $u(t,x;s)$ behind some special interface.

\begin{lem}\label{lem-hehind-key}
There exists $\la_{0}\in(\theta,1)$, $r>0$ and $\beta_{0}>0$ such that
\begin{equation*}
u(t,x+\xi_{\la_{0}}(t;s);s)\geq1-(1-\la_{0})\Big[e^{-\be_{0}(t-s)}+e^{rx}\Big],\quad x\leq0
\end{equation*}
for all $s<0$, $t\geq s+T_{D}$
\end{lem}
\begin{proof}
By $\rm(H2)$, there exist $\la_{0}\in(\theta,1)$ and $\be_{0}>0$ such that
\begin{equation}\label{condition-lower-estimate}
f(t,u)\geq\be_{0}(1-u),\quad u\in[\la_{0},1].
\end{equation}
Let $v(t,x;s)=u(t,x+\xi_{\la_{0}}(t;s);s)$. It solves
\begin{equation*}
\begin{split}
\left\{\begin{aligned}
&v_{t}=v_{xx}+\xi_{\la_{0}}'v_{x}+f(t,v),\,\,x\leq0,\,\,t\geq s+T_{D}\\
&v(t,-\infty;s)=1,\,\,v(t,0;s)=\la_{0},\,\,t\geq s+T_{D}\\
&v(s+T_{D},x;s)=u(s+T_{D},x+\xi_{\la_{0}}(s+T_{D};s);s),\,\,x\leq0
\end{aligned} \right.
\end{split}
\end{equation*}
where $\xi_{\la_{0}}'=\frac{d}{dt}\xi_{\la_{0}}(t;s)$. Since $v(t,x;s)\in[\la_{0},1)$ for $x\leq0$ and $t\geq s+T_{D}$, we conclude from \eqref{condition-lower-estimate} that
\begin{equation}\label{estimate-hehind-1}
v(t,x;s)\geq\hat{v}(t,x;s),\quad x\leq0,\,\,t\geq s+T_{D},
\end{equation}
where $\hat{v}(t,x;s)$ is the solution of
\begin{equation*}
\begin{split}
\left\{\begin{aligned}
&\hat{v}_{t}=\hat{v}_{xx}+\xi_{\la_{0}}'\hat{v}_{x}+\beta_{0}(1-\hat{v}),\,\,x\leq0,\,\,t\geq s+T_{D}\\
&\hat{v}(t,-\infty;s)=1,\,\,\hat{v}(t,0;s)=\la_{0},\,\,t\geq s+T_{D}\\
&\hat{v}(s+T_{D},x;s)=u(s+T_{D},x+\xi_{\la_{0}}(s+T_{D};s);s),\,\,x\leq0.
\end{aligned} \right.
\end{split}
\end{equation*}

Let $C_{\la_{0}}$ be as in Theorem \ref{thm-non-flat} $\rm(ii)$ and $\tilde{v}(x)$, $x\leq0$ be the solution of
\begin{equation*}
\begin{split}
\left\{\begin{aligned}
&\tilde{v}_{xx}+C_{\la_{0}}\tilde{v}_{x}+\beta_{0}(1-\tilde{v})=0,\,\,x\leq0\\
&\tilde{v}(-\infty)=1,\,\,\tilde{v}(0)=\la_{0}.
\end{aligned} \right.
\end{split}
\end{equation*}
The above problem is explicitly solvable, and we readily compute
\begin{equation}\label{estimate-hehind-2}
\tilde{v}(x)=1-(1-\la_{0})e^{rx},\quad x\leq0,
\end{equation}
where $r=\frac{-C_{\la_{0}}+\sqrt{C_{\la_{0}}^{2}+4\beta_{0}}}{2}>0$. Setting
\begin{equation}\label{estimate-hehind-3}
\bar{v}(t,x;s)=\hat{v}(t,x;s)-\tilde{v}(x),\,\,x\leq0,\,\,t\geq s+T_{D}
\end{equation}
we easily check that $\bar{v}(t,x;s)$ satisfies
\begin{equation*}
\begin{split}
\left\{\begin{aligned}
&\bar{v}_{t}\geq\bar{v}_{xx}+\xi_{\la_{0}}'\bar{v}_{x}-\beta_{0}\bar{v},\,\,x\leq0,\,\,t\geq s+T_{D}\\
&\bar{v}(t,0;s)=0,\,\,\bar{v}(t,-\infty;s)=0,\,\,t\geq s+T_{D}\\
&\bar{v}(s+T_{D},x;s)=u(s+T_{D},x+\xi_{\la_{0}}(s+T_{D};s);s)-\tilde{v}(x),\,\,x\leq0.
\end{aligned} \right.
\end{split}
\end{equation*}
Since clearly $\bar{v}(s,x;s)\geq\la_{0}-1$, we obtain that
\begin{equation}\label{estimate-hehind-4}
\bar{v}(t,x;s)\geq(\la_{0}-1)e^{-\be_{0}(t-s)},\quad x\leq0,\,\,t\geq s+T_{D}
\end{equation}
where $(\la_{0}-1)e^{-\be_{0}(t-s)}$ is a space-independent solution of $\bar{v}_{t}=\bar{v}_{xx}+\xi_{\la_{0}}'\bar{v}_{x}-\beta_{0}\bar{v}$. The result then follows from \eqref{estimate-hehind-1}, \eqref{estimate-hehind-2}, \eqref{estimate-hehind-3} and \eqref{estimate-hehind-4}.
\end{proof}

We now prove Theorem \ref{thm-uniform-estimate-hehind}.

\begin{proof}[Proof of Theorem \ref{thm-uniform-estimate-hehind}]
$\rm(i)$ For $\la\in[\theta,1)$, we define
\begin{equation*}
M_{\la}=\sup_{s<0,t\geq s+T_{D}}\big[\xi_{\theta}(t;s)-\xi_{\la}(t;s)\big].
\end{equation*}
Clearly, $M_{\la}\leq C(\theta,\la)$ by Theorem \ref{thm-bd-interface-width} with the understanding $C(\theta,\theta)=0$, $M_{\la}\ra\infty$ as $\la\ra1$ by just looking at $u(t_{0},x;s_{0})$ for some $s_{0}<0$ and $t_{0}\geq s_{0}+T_{D}$, and the map $\la\mapsto M_{\la}:[\theta,1)\ra[0,\infty)$ is nondecreasing. We show that $\la\mapsto M_{\la}:[\theta,1)\ra[0,\infty)$ is strictly increasing and continuous.

We show that $\la\mapsto M_{\la}:[\theta,1)\ra[0,\infty)$ is strictly increasing. Fix any $\la_{0}\in[\theta,1)$ and let $\{s_{n}+T_{D}\leq t_{n}\}_{n\in\N}$ be such that $\lim_{n\ra\infty}\xi_{\theta}(t_{n};s_{n})-\xi_{\la_{0}}(t_{n};s_{n})=M_{\la_{0}}$. Since $\inf_{s<0,t\geq s, x\in\R}u_{x}(t,x;s)\geq-C^{*}$ for some $C^{*}>0$ by a priori estimates parabolic equations, we find
\begin{equation}\label{estimate-upper-bd-aux}
\begin{split}
u(t_{n},x+\xi_{\theta}(t_{n};s_{n});s_{n})\leq\min\Big\{-C^{*}\big[x-(\xi_{\la_{0}}(t_{n};s_{n})-\xi_{\theta}(t_{n};s_{n}))\big]+\la_{0},1\Big\}
\end{split}
\end{equation}
for all $x\in(-\infty,\xi_{\la_{0}}(t_{n};s_{n})-\xi_{\theta}(t_{n};s_{n})]$ and all $n\in\N$.

Now, let $\la_{1}\in(\la_{0},1)$. Using \eqref{estimate-upper-bd-aux} and $\lim_{n\ra\infty}\xi_{\theta}(t_{n};s_{n})-\xi_{\la_{0}}(t_{n};s_{n})=M_{\la_{0}}$, we can find an $N$ sufficiently large such that
$\xi_{\la_{1}}(t_{N};s_{N})-\xi_{\theta}(t_{N};s_{N})\leq x_{N}$, where $x_{N}<-M_{\la_{0}}$ is such that $-(C^{*}+1)(x_{N}+M_{\la_{0}})+\la_{0}=\la_{1}$. It then follows
\begin{equation*}
M_{\la_{1}}\geq\xi_{\theta}(t_{N};s_{N})-\xi_{\la_{1}}(t_{N};s_{N})\geq-x_{N}>M_{\la_{0}}.
\end{equation*}

We show that $\la\mapsto M_{\la}:[\theta,1)\ra[0,\infty)$ is continuous. Fix any $\la_{0}\in[\theta,1)$ and let $\ep_{0}>0$ be small. By Theorem \ref{thm-non-flat}, there is $\al_{0}>0$ such that
\begin{equation*}
\sup_{s<0,t\geq s+T_{D}\atop x\in[\xi_{\la_{0}}(t;s)-\xi_{\theta}(t;s)-\ep_{0},\xi_{\la_{0}}(t;s)-\xi_{\theta}(t;s)+\ep_{0}]}u_{x}(t,x+\xi_{\theta}(t;s);s)\leq-\al_{0}.
\end{equation*}
It follows that for all $s<0$, $t\geq s+T_{D}$
\begin{equation}\label{continuity-right}
u(t,x+\xi_{\theta}(t;s);s)\geq-\al_{0}\big[x-(\xi_{\la_{0}}(t;s)-\xi_{\theta}(t;s))\big]+\la_{0}
\end{equation}
for $x\in[\xi_{\la_{0}}(t;s)-\xi_{\theta}(t;s)-\ep_{0},\xi_{\la_{0}}(t;s)-\xi_{\theta}(t;s)]$, and
\begin{equation}\label{continuity-left}
u(t,x+\xi_{\theta}(t;s);s)\leq-\al_{0}\big[x-(\xi_{\la_{0}}(t;s)-\xi_{\theta}(t;s))\big]+\la_{0}
\end{equation}
for $x\in[\xi_{\la_{0}}(t;s)-\xi_{\theta}(t;s),\xi_{\la_{0}}(t;s)-\xi_{\theta}(t;s)+\ep_{0}]$.

Then, comparing \eqref{continuity-right} with the segment $-\al_{0}(x+M_{\la_{0}})+\la_{0}$ for $x\in[-M_{\la_{0}}-\ep_{0},-M_{\la_{0}}]$, we obtain for any $\la\in(\la_{0},\al_{0}\ep_{0}+\la_{0}]$, $\xi_{\la}(t;s)-\xi_{\theta}(t;s)\geq-M_{\la_{0}}-\frac{\la-\la_{0}}{\al_{0}}$ for all $s<0$, $t\geq s+T_{D}$, which together with the fact that $\la\mapsto M_{\la}:[\theta,1)\ra[0,\infty)$ is strictly increasing implies
\begin{equation*}
M_{\la_{0}}<M_{\la}\leq M_{\la_{0}}+\frac{\la-\la_{0}}{\al_{0}}\ra M_{\la_{0}}\,\,\text{as}\,\,\la\ra\la_{0}^{+}.
\end{equation*}
This show the right continuity at $\la_{0}$.

For the left continuity, for any $\la\in[\la_{0}-\frac{1}{2}\al_{0}\ep_{0},\la_{0})$, we pick a sequence $\{s_{n}+T_{D}\leq t_{n}\}_{n\in\N}$ such that
$$
\lim_{n\ra\infty}\xi_{\theta}(t_{n};s_{n})-\xi_{\la_{0}}(t_{n};s_{n})=M_{\la_{0}}.
$$
Then, comparing \eqref{continuity-left} with the segment $-\frac{1}{2}\al_{0}(x+M_{\la_{0}})+\la_{0}$ for $x\in[-M_{\la_{0}},-M_{\la_{0}}+\ep_{0}]$, we can find an $N$ sufficiently large such that
\begin{equation*}
\xi_{\la}(t_{N};s_{N})-\xi_{\theta}(t_{N};s_{N})\leq-M_{\la_{0}}-\frac{2(\la-\la_{0})}{\al_{0}},
\end{equation*}
which implies $M_{\la}\geq M_{\la_{0}}+\frac{2(\la-\la_{0})}{\al_{0}}$, and then the left continuity at $\la_{0}$.

So far, we have shown that  $\la\mapsto M_{\la}:[\theta,1)\ra[0,\infty)$ is strictly increasing, continuous, and satisfies $M_{\theta}=0$ and $M_{\la}\ra+\infty$ as $\la\ra1$. We now define $v:(-\infty,0]\ra[\theta,1)$ as the inverse function of $\la\mapsto-M_{\la}$. It is easily verified that this $v$ satisfies all required properties as in the statement.

$\rm(ii)$ By  Lemma \ref{lem-hehind-key}$\rm(ii)$, we have
\begin{equation*}
u(t,x+\xi_{\la_{0}}(t;s);s)\geq1-(1-\la_{0})\Big[e^{-\be_{0}(t-s)}+e^{r[x-(\xi_{\la_{0}}(t;s)-\xi_{\theta}(t;s))]}\Big]
\end{equation*}
for $x\leq\xi_{\la_{0}}(t;s)-\xi_{\theta}(t;s)$. Since $\xi_{\la_{0}}(t;s)-\xi_{\theta}(t;s)\geq- C(\theta,\la_{0})$ by Theorem \ref{thm-bd-interface-width}, we arrive at the result.
\end{proof}

\subsection{Uniform Decaying Estimates Ahead of Interface}\label{subsec-estimate-ahead-interface}

In this subsection, we control $u(t,x;s)$ ahead of the interface. The main result of this subsection is stated in the following theorem.

\begin{thm}\label{thm-decaying-2}
There exist $\hat{T}_{D}>0$ and $c>0$ such that
\begin{equation*}
u(t,x+\xi_{\theta}(t;s);s)\leq\theta e^{-cx},\quad x\geq0
\end{equation*}
for all $s<0$ and $t\geq s+\hat{T}_{D}$.
\end{thm}

This theorem says that $u(t,x;s)$ decays from the interface with a uniform decaying rate.  It actually contains much more information than it looks like. For example, since $u(t,\xi_{\theta}(t;s);s)=\theta$, Theorem \ref{thm-decaying-2} then implies $u_{x}(t,\xi_{\theta}(t;s);s)\leq-c\theta$, although we have obtained this information in Theorem \ref{thm-non-flat}.

To prove Theorem \ref{thm-decaying-2}, we first prove several lemmas. The first one concerns the rightward propagation of $\xi_{\theta}(t;s)$.

\begin{lem}\label{lem-propogation}
There exist $T_{*}>0$ and $h_{*}>0$ such that
\begin{equation*}
\xi_{\theta}(t+T_{*};s)-\xi_{\theta}(t;s)\geq h_{*}
\end{equation*}
for all $s<0$, $t\geq s+T_{D}$.
\end{lem}
\begin{proof}
The lemma follows from Theorem \ref{thm-bd-interface-width} and Lemma \ref{lem-rightward-prop-above-temp}.
\end{proof}

The next lemma is the driving force for the so-called sliding method (see \cite{BeNi91}).
\begin{lem}\label{lem-sliding}
Let $c\in(0,\min\{c_{\inf},\frac{h_{*}}{T_{*}}\})$ and $s<0$, where $T_{*}$ and $h_{*}$ are as in Lemma \ref{lem-propogation}. Suppose there is $t_{*}\geq s+T_{D}$ such that $u(t_{*},x+\xi_{\theta}(t_{*};s);s)\leq\theta e^{-cx}$ for $x\geq0$. Then, there exists $T(t_{*})\in(t_{*},\infty)$ such that
\begin{equation*}
u(T(t_{*}),x+\xi_{\theta}(T(t_{*});s);s)\leq\theta e^{-cx},\quad x\geq0.
\end{equation*}
Moreover, there are constants $0<c_{0}<C_{0}$ (independent of $s$ and $t_{*}$) such that
\begin{equation*}
c_{0}\leq T(t_{*})-t_{*}\leq C_{0}.
\end{equation*}
\end{lem}
\begin{proof}
Fix some $\theta_{*}\in(\theta,1)$. For $t\geq t_{*}$, define
\begin{equation}\label{aux-fun-sliding}
v(t,x;t_{*})=\theta_{*}e^{-c(x-\xi_{\theta}(t_{*};s)-c(t-t_{*}))}.
\end{equation}
Clearly, $u(t_{*},x;s)<v(t_{*},x;t_{*})$ for $x\geq\xi_{\theta}(t_{*};s)$ by assumption. By comparison principle, we have $u(t,x;s)<v(t,x;t_{*})$ for $x\geq\xi_{\theta}(t;s)$ for all $t>t_{*}$ with $t-t_{*}$ sufficiently small. In fact, since $u(t_{*},\xi_{\theta}(t_{*};s);s)<v(t_{*},\xi_{\theta}(t_{*};s);t_{*})$, continuity ensures the existence of some $t_{1}>t_{*}$ with $t_{1}-t_{*}$ small such that $u(t,\xi_{\theta}(t;s);s)<v(t,\xi_{\theta}(t;s);t_{*})$ for all $t\in[t_{*},t_{1}]$. Since $v(t,x;t_{*})$ solves $v_{t}=v_{xx}$ and $f(t,u(t,x;s))=0$ for $x\geq\xi_{\theta}(t;s)$, we conclude from the comparison principle that $u(t,x;s)<v(t,x;t_{*})$ for $x\geq\xi_{\theta}(t;s)$ for all $t\in[t_{*},t_{1}]$.

Now, we define
\begin{equation*}
T(t_{*})=\sup\big\{t\geq t_{*}\big|u(\tau,x;s)<v(\tau,x;t_{*}),\,\,x\geq\xi_{\theta}(\tau;s)\,\,\text{holds for all}\,\,\tau\in[t_{*},t)\big\}.
\end{equation*}
Clearly, $T(t_{*})>t_{*}$. Since $\phi(x-x_{s}-c_{\inf}(t-s))\leq u(t,x;s)$ and $c<c_{\inf}$, we conclude that $T(t_{*})<\infty$.

Again, since $v(t,x;t_{*})$ solves $v_{t}=v_{xx}$ and $f(t,u(t,x;s))=0$ for $x\geq\xi_{\theta}(t;s)$, we conclude from the comparison principle that, at time $T(t_{*})$, we must have
\begin{equation*}
\begin{split}
&u(T(t_{*}),x;s)\leq v(T(t_{*}),x;s),\quad x\geq\xi_{\theta}(T(t_{*});s),\\
&u(T(t_{*}),\xi_{\theta}(T(t_{*});s);s)=\theta=v(T(t_{*}),\xi_{\theta}(T(t_{*});s);t_{*}).
\end{split}
\end{equation*}
Using \eqref{aux-fun-sliding}, we readily check $u(T(t_{*}),x;s)\leq\theta e^{-c(x-\xi_{\theta}(T(t_{*});s))}$ for $x\geq\xi_{\theta}(T(t_{*});s)$.

For the ``moreover" part, let $\eta_{\theta}(t;t_{*})$ be the unique point such that $v(t,\eta_{\theta}(t;t_{*});t_{*})=\theta$. Then, $T(t_{*})$ is the first time that $\xi_{\theta}(t;s)$ hits $\eta_{\theta}(t;t_{*})$. Note that $\eta_{\theta}(t;t_{*})$ moves rightward at a constant speed $c$, that is,
\begin{equation*}
\eta_{\theta}(t;t_{*})=\eta_{\theta}(t_{*};t_{*})+c(t-t_{*})=\xi_{\theta}(t_{*};s)+\frac{1}{c}\ln\frac{\theta_{*}}{\theta}+c(t-t_{*}).
\end{equation*}
By Lemma \ref{lem-propogation}, for any $n\in\N$, $\xi_{\theta}(t_{*}+nT_{*};s)\geq \xi_{\theta}(t_{*};s)+nh_{*}$. Since $c<\frac{h_{*}}{T_{*}}$, we can find some $n_{0}$ such that
$\xi_{\theta}(t_{*}+n_{0}T_{*};s)\geq\eta_{\theta}(t_{*}+n_{0}T_{*};t_{*})$, which leads to $T(t_{*})-t_{*}\leq n_{0}T_{*}$. This establishes the upper bound.

For the lower bound, we use Theorem \ref{thm-non-flat}(ii), saying that $\xi_{\theta}(t;s)$ propagates not faster than the speed $C_{*}:=C_{\theta}\geq c_{\inf}$. Therefore, it takes, at least, $\frac{\eta_{\theta}(t_{*};t_{*})-\xi_{\theta}(t_{*};s)}{C_{*}-c}=\frac{1}{c(C_{*}-c)}\ln\frac{\theta_{*}}{\theta}$, for $\xi_{\theta}(t;s)$ to hit $\eta_{\theta}(t;t_{*})$. Thus, $T(t_{*})-t_{*}\geq\frac{1}{c(C_{*}-c)}\ln\frac{\theta_{*}}{\theta}$.
\end{proof}

We remark that the constant $c_{0}$ in the statement of Lemma \ref{lem-sliding} does depend on the choice of $c$ as in the statement of the lemma and $\theta_{*}$ as in the proof. But this will not cause any trouble, because we only need some $c\in(0,\min\{c_{\inf},\frac{h_{*}}{T_{*}}\})$ and some $\theta_{*}\in(\theta,1)$.

Lemma \ref{lem-sliding} lays the foundation for an iteration argument. To run such an argument, we need the exponential decay condition as in the lemma to hold at some initial time greater than $s+T_{D}$. This is given by

\begin{lem}\label{lem-sliding-initial}
Let $c\in(0,\min\{c_{\inf},\frac{h_{*}}{T_{*}}\})$ be small. For any $s<0$ there exists $T_{s}>0$ such that
\begin{equation*}
u(s+T_{s},x+\xi_{\theta}(s+T_{s};s);s)\leq\theta e^{-cx},\quad x\geq0.
\end{equation*}
Moreover, $T_{D}\leq T_{s}\leq\hat{C}_{0}$ for some $\hat{C}_{0}>0$.
\end{lem}
\begin{proof}
Fix some $\theta_{*}\in(\theta,1)$. By Proposition \ref{prop-bd-propagation} and \eqref{estimate-lower-bound}, there exists $h_{D}>0$ such that
\begin{equation}\label{estimate-possible-position}
\xi_{\theta}(t;s)\in[x_{s},x_{s}+h_{D}]\,\,\text{for}\,\,t\in[s,s+T_{D}].
\end{equation}
Now, for $t\geq s$, we define
\begin{equation}\label{aux-fun-sliding-intial}
v(t,x;s)=\theta_{*}e^{-c(x-x_{s}-c(t-s))}.
\end{equation}
Note that for small $c$, we can guarantee that the unique solution of the algebraic equation $\theta_{*}e^{-c(x-x_{s})}=\theta$ is greater than $x_{s}+h_{D}$. Let us denote this solution by $x_{s}+h_{D}+x_{D}$ for some $x_{D}>0$. As in Lemma \ref{lem-sliding}, let $\eta_{\theta}(t;s)$ be the unique point such that $v(t,\eta_{\theta}(t;s);s)=\theta$ and $s+T_{s}$ be the first time that $\xi_{\theta}(t;s)$ hits $\eta_{\theta}(t;s)$.

Since $\eta_{\theta}(s;s)=x_{s}+h_{D}+x_{D}$ and $\eta_{\theta}(t;s)$ moves rightward at a constant speed $c$, \eqref{estimate-possible-position} ensures $T_{s}\geq T_{D}$. On the other hand, by \eqref{explicit-sol} and \eqref{estimate-lower-bound}, we have $\xi_{\theta}(t;s)\geq x_{s}+c_{\inf}(t-s)$, which implies that it will take, at most, $\frac{h_{D}+x_{D}}{c_{\inf}-c}$, for $\xi_{\theta}(t;s)$ to hit $\eta_{\theta}(t;s)$. Thus, $T_{s}\leq\frac{h_{D}+x_{D}}{c_{\inf}-c}$.

Finally, at the first hitting time $s+T_{s}$, we have the estimate
\begin{equation*}
u(s+T_{s},x+\xi_{\theta}(s+T_{s};s);s)\leq\theta e^{-cx},\quad x\geq0.
\end{equation*}
as in the proof of Lemma \ref{lem-sliding}.
\end{proof}

Finally, we prove Theorem \ref{thm-decaying-2}.

\begin{proof}[Proof of Theorem \ref{thm-decaying-2}]
Let $c\in(0,\min\{c_{\inf},\frac{h_{*}}{T_{*}}\})$ be small such that both Lemma \ref{lem-sliding} and Lemma \ref{lem-sliding-initial} hold. By Lemma \ref{lem-sliding-initial}, we have $u(s+T_{s},x+\xi_{\theta}(s+T_{s};s);s)\leq\theta e^{-cx}$ for $x\geq0$. Since $T_{s}\geq T_{D}$ by Lemma \ref{lem-sliding-initial}, we can apply Lemma \ref{lem-sliding} to obtain that at each moment $T^{n}(s+T_{s})=\underbrace{T\circ T\circ\cdots\circ T}_{n\,\,\text{times}}(s+T_{s})$, there holds
\begin{equation*}
u(T^{n}(s+T_{s}),x+\xi_{\theta}(T^{n}(s+T_{s});s);s)\leq\theta e^{-cx},\quad x\geq0
\end{equation*}
for all $n\in\N$. Again, by Lemma \ref{lem-sliding}, $c_{0}\leq T^{n}(s+T_{s})-T^{n-1}(s+T_{s})\leq C_{0}$ for all $n\in\N$. In particular, $T^{n}(s+T_{s})\ra\infty$ as $n\ra\infty$, and $[s+T_{s},\infty)=\cup_{n\in\N}[T^{n-1}(s+T_{s}),T^{n}(s+T_{s})]$.

Next, we claim that there is $\hat{\theta}>0$ such that
\begin{equation}\label{estimate-decaying}
u(t,x+\xi_{\theta}(t;s);s)\leq\hat{\theta}e^{-cx},\quad x\geq0
\end{equation}
for all $s<0$, $t\geq s+T_{s}$. Fix any $n\in\N$ and consider the interval $[T^{n-1}(s+T_{s}),T^{n}(s+T_{s})]$. By the proof of Lemma \ref{lem-sliding}, we have for any $t\in[T^{n-1}(s+T_{s}),T^{n}(s+T_{s})]$ and $x\geq\xi_{\theta}(t;s)$,
\begin{equation*}
\begin{split}
u(t,x;s)&\leq v(t,x;T^{n-1}(s+T_{s}))\\
&=\theta_{*}e^{-c(x-\xi_{\theta}(T^{n-1}(s+T_{s});s)-c(t-T^{n-1}(s+T_{s})))}\\
&=\theta_{*}e^{-c(x-\xi_{\theta}(t;s))}e^{-c(\xi_{\theta}(t;s)-\xi_{\theta}(T^{n-1}(s+T_{s});s))}e^{c^{2}(t-T^{n-1}(s+T_{s}))}\\
&\leq\theta_{*}e^{cC_{*}(t-T^{n-1}(s))}e^{c^{2}(t-T^{n-1}(s+T_{s}))}e^{-c(x-\xi_{\theta}(t;s))}\\
&\leq\theta_{*}e^{cC_{*}C_{0}}e^{c^{2}C_{0}}e^{-c(x-\xi_{\theta}(t;s))}.
\end{split}
\end{equation*}
The claim follows with $\hat{\theta}=\theta_{*}e^{cC_{*}C_{0}}e^{c^{2}C_{0}}$, where $C_{*}=C_{\theta}$ is given in Theorem \ref{thm-non-flat}(ii).

To finish the proof, we fix some $M_{0}>0$ and set $\al_{M_{0}}=\min_{M\in[0,M_{0}]}\al(M)$, where the function $\al(\cdot)$ is given by Theorem \ref{thm-non-flat}. It then follows from the fact that $u(t,\xi_{\theta}(t;s);s)=\theta$ and Theorem \ref{thm-non-flat} that
\begin{equation*}
u(t,x+\xi_{\theta}(t;s);s)\leq-\al_{M_{0}}x+\theta,\quad x\in[0,M_{0}]
\end{equation*}
for $s<0$, $t\geq s+T_{s}$. By monotonicity, we obtain $u(t,x+\xi_{\theta}(t;s);s)\leq-\al_{M_{0}}M_{0}+\theta$ for $x\geq M_{0}$. Note that by enlarging $\hat{\theta}$ if necessary, we may assume without loss of generality that $\hat{\theta}e^{-cM_{0}}>\theta$, which implies $-\al_{M_{0}}x+\theta<\hat{\theta}e^{-cx}$ for all $x\in[0,M_{0}]$. Now, let $x_{*}>M_{0}$ be the smallest point such that $\hat{\theta}e^{-cx_{*}}=-\al_{M_{0}}x+\theta$. All these together, we obtain for $x\geq0$
\begin{equation*}
\begin{split}
u(t,x+\xi_{\theta}(t;s);s)\leq\psi_{*}(x)=\left\{\begin{aligned}
-\al_{M_{0}}x+\theta,&\quad x\in[0,M_{0}],\\
-\al_{M_{0}}M_{0}+\theta,&\quad x\in[M_{0},x_{*}],\\
\hat{\theta}e^{-cx},&\quad x\geq x_{*}.
\end{aligned} \right.
\end{split}
\end{equation*}
Set $c_{*}=\frac{1}{x_{*}}\ln\frac{\theta}{\theta-\al_{M_{0}}M_{0}}$, that is, $\theta e^{-c_{*}x_{*}}=-\al_{M_{0}}M_{0}+\theta$. By further enlarging $\hat{\theta}$ if necessary, we can make $x_{*}$ sufficiently large so that $c_{*}\leq\al_{M_{0}}$, which ensures $\phi_{*}(x)\leq\theta e^{-c_{*}x}$ for $x\geq0$. Hence, $u(t,x+\xi_{\theta}(t;s);s)\leq\theta e^{-c_{*}x}$ for $x\geq0$. The theorem then follows with $\hat{T}_{D}=\sup_{s<0}T_{s}$.
\end{proof}


\section{Transition Fronts in Time Heterogeneous Media}\label{sec-transition-wave-const}

In this section, we investigate front propagation phenomena in \eqref{main-eqn} and prove Theorem \ref{thm-transition-wave}. Throughout this section, we assume $\rm(H1)$ and $\rm(H2)$. We first present two lemmas about critical transition fronts (see Definition \ref{def-transition-wave}).

\begin{lem}[Uniqueness of critical transition fronts]
\label{lm-critical-traveling-wave1}
If $u(t,x)$ and $\tilde{u}(t,x)$ are critical transition fronts of \eqref{main-eqn}, then there is a space shift $\zeta_{0}\in\R$ such that $u(t,x+\zeta_{0})=\tilde{u}(t,x)$ for all $x\in\R$ and $t\in\R$.
\end{lem}
\begin{proof}
It follows from the arguments of \cite[Theorem A]{Sh04}.
\end{proof}

\begin{lem}[Existence of critical transition fronts]
\label{lm-critical-traveling-wave2}
If \eqref{main-eqn} admits a transition front $u(t,x)$, then it admits a critical transition front $u^{c}(t,x)$.
\end{lem}
\begin{proof}
It follows from the arguments of \cite[Theorem A]{Sh04}.
\end{proof}

We now prove Theorem \ref{thm-transition-wave}.

\begin{proof}[Proof of Theorem \ref{thm-transition-wave}]
$\rm(1)$ Let $u(t,x)$, $x\in\R$, $t\in\R$ be the global-in-time solution of \eqref{main-eqn} given in Theorem \ref{thm-entire-sol}.

We first show that there is a continuously differentiable function $\xi:\R\ra\R$ such that $u(t,\xi(t))=\theta$ for all $t\in\R$. Since
\begin{equation*}
\sup_{t<s,t\geq s+T_{D}}\bigg|\frac{d}{dt}\xi_{\theta}(t;s)\bigg|\leq C_{*}
\end{equation*}
by Theorem \ref{thm-non-flat}, we conclude from Arzel\`{a}-Ascoli theorem and the diagonal argument that $\xi_{\theta}(t;s)$ converges to $\xi(t)$ uniformly on any compact set as $s\ra-\infty$ along some subsequence. Since $u(t,\xi_{\theta}(t;s);s)=\theta$ for all $s<0$, $t\geq s$, we find $u(t,\xi(t))=\theta$ for all $t\in\R$. By Theorem \ref{thm-bd-interface-width}, $u(t,x)$ is a transition front of \eqref{main-eqn}.

$\rm(1)(i)$ Since $\lim_{x\ra-\infty}u(t,x)=1$ and $\lim_{x\ra\infty}u(t,x)=0$, $u(t,x)$ is strictly decreasing on some open set. We now fix some $t_{0}$ as an initial moment and consider the solution $u(t,x)$ for $t\geq t_{0}$. Let $y>0$. Since $u(t_{0},x+y)-u(t_{0},x)\leq0$ for all $x$ and $u(t_{0},x+y)-u(t_{0},x)<0$ on some open set, we apply maximum principle to $u(t,x+y)-u(t,x)$ to conclude that $u(t,x+y)-u(t,x)<0$ for all $x\in\R$ and $t>t_{0}$. Since $u(t,x)$ is a global-in-time solution, $u(t,x)$ is strictly decreasing in $x$ for all $t\in\R$. We then conclude $u_{x}(t,x)<0$ from Angenent's result (see \cite[Theorem B]{Ang88}) as in the proof of Lemma \ref{lem-basic-prop}.

$\rm(1)(ii)$ For continuous differentiability, we first use the limit $u_{x}(t,\xi_{\theta}(t;s);s)\ra u_{x}(t,\xi(t))$ as $s\ra\infty$ along some subsequence and Theorem \ref{thm-non-flat} to conclude that $\sup_{t\in\R}u_{x}(t,\xi(t))<0$.
The result then follows from the arguments as in the proof of Lemma \ref{lem-negativity}. In particular, we have
\begin{equation*}
\xi'(t)=-\frac{u_{t}(t,\xi(t))}{u_{x}(t,\xi(t))},\quad t\in\R.
\end{equation*}
As a byproduct, we also have $\sup_{t\in\R}|\xi'(t)|<\infty$.

$\rm(1)(iii)$ Since $u(t,x+\xi_{\theta}(t;s);s)\geq v(x)$ for all $x\leq0$, $s<0$ and $t\geq s+T_{D}$ by Theorem \ref{thm-uniform-estimate-hehind}$\rm(i)$, we have $u(t,x+\xi(t))\geq v(x)$ for $x\leq0$ and $t\in\R$. Moreover, setting $s\ra-\infty$ along some subsequence in the estimate
\begin{equation*}
u(t,x+\xi_{\theta}(t;s);s)\geq1-(1-\la_{0})\Big[e^{-\be_{0}(t-s)}+e^{r(x+C(\theta,\la_{0}))}\Big],\quad x\leq-C(\theta,\la_{0})
\end{equation*}
for $s<0$, $t\geq s+T_{D}$ given by Theorem \ref{thm-uniform-estimate-hehind}$\rm(ii)$, we conclude that
\begin{equation*}
u(t,x+\xi(t))\geq1-(1-\la_{0})e^{r(x+C(\theta,\la_{0}))},\quad x\leq-C(\theta,\la_{0}).
\end{equation*}
That is, $1-u(t,x+\xi(t))$ decays exponentially as $x\ra-\infty$ and the decay is uniform in $t\in\R$.

By Theorem \ref{thm-decaying-2}, we clearly have $u(t,\xi(t))\leq\theta e^{-cx}$ for $x\geq0$ and $t\in\R$. Thus, by setting
\begin{equation*}
\begin{split}
\hat{v}(x)=\left\{\begin{aligned}
\max\Big\{v(x),1-(1-\la_{0})e^{r(x+C(\theta,\la_{0}))}\Big\},\quad x\leq0,\\
\theta e^{-cx},\quad x\geq0,\\
\end{aligned} \right.
\end{split}
\end{equation*}
we find the function satisfying all required properties.

$\rm(2)$ By (1) and Lemma \ref{lm-critical-traveling-wave2}, \eqref{main-eqn} has a critical transition front $u^{c}(t,x)$. We prove that it must be a periodic traveling wave. Clearly, $u^{c}(\cdot+T,\cdot)$ is also a transition front. We show its criticality. Let $u$ be an arbitrary transition front. Then $u(\cdot-T,\cdot)$ is a transition front as well. Thus, for any $t\in\R$, there is a $\zeta(t)\in\R$ such that $u^{c}(t,x)\geq u(t-T,x)$ if $x\leq\zeta(t)$ and $u^{c}(t,x)\leq u(t-T,x)$ if $x\geq\zeta(t)$. Replacing $t$ by $t+T$, we find for any $t\in\R$, $u^{c}(t+T,x)\geq u(t,x)$ if $x\leq\zeta(t+T)$ and $u^{c}(t+T,x)\leq u(t,x)$ if $x\geq\zeta(t+T)$. Hence, $u^{c}(\cdot+T,\cdot)$ is critical.

By Lemma  \ref{lm-critical-traveling-wave1}, there exists some $\zeta_{0}\in\R$ such that
\begin{equation}\label{equality-periodic}
u^{c}(t,x+\zeta_{0})=u^{c}(t+T,x),\quad x\in\R,\,\,t\in\R.
\end{equation}
For $t\in\R$, let $\xi_{c}(t)$ be the unique point such that $u^{c}(t,\xi_{c}(t))=\theta$. Assume, without loss of generality, that $\xi_{c}(0)=0$. Setting $t=0$ and $x=\xi_{c}(T)$ in \eqref{equality-periodic}, we find $u^{c}(0,\xi_{c}(T)+\zeta_{0})=u^{c}(T,\xi(T))=\theta$. It follows $\xi_{c}(T)+\zeta_{0}=\xi_{c}(0)=0$, and hence, $\zeta_{0}=-\xi_{c}(T)$. Thus,
\begin{equation}\label{equality-periodic-1}
u^{c}(t,x-\xi_{c}(T))=u^{c}(t+T,x),\quad x\in\R,\,\,t\in\R.
\end{equation}
Let $c_{T}=\frac{\xi_{c}(T)}{T}$. Define
\begin{equation*}
\psi(t,x)=u^{c}(t,x+c_{T}t),\quad x\in\R,\,\,t\in\R.
\end{equation*}
We check $\psi$ satisfies the properties required by a profile of a periodic traveling wave. Using \eqref{equality-periodic-1}, we have for $x\in\R$,
\begin{equation*}
\begin{split}
\psi(t+T,x)&=u^{c}(t+T,x+c_{T}(t+T))\\
&=u^{c}(t+T,x+c_{T}t+\xi_{c}(T))\\
&=u^{c}(t,x+c_{T}t)\\
&=\psi(t,x),
\end{split}
\end{equation*}
that is, $\psi(\cdot+T,\cdot)=\psi$. The uniform-in-time limit at $\pm\infty$ then follows. Since $u^{c}$ solves \eqref{main-eqn}, we readily check $\psi_{t}=\psi_{xx}+c_{T}\psi_{x}+f(t,\psi)$. In conclusion, $u^{c}$ is a periodic traveling wave.
\end{proof}


\section{Transition Fronts in Random Media}\label{sec-random-wave}

In this section, we explore front propagation phenomena in \eqref{random-eq} and prove Theorem \ref{thm-random-wave}.

\begin{proof} [Proof of Theorem \ref{thm-random-wave}]
$\rm(1)$ First of all, for any fixed  $\omega\in\Omega$, by Theorem \ref{thm-transition-wave}, \eqref{random-eq} admits a transition front $u^\omega(t,x)$. Clearly, $u^\omega(t,x)$ is a wave-like solution of \eqref{random-eq} in the sense of \cite[Definition 2.3]{Sh04}. Then by \cite[Theorem A (1)]{Sh04},
\eqref{random-eq} admits a random traveling wave solution $u(t,x;\omega)$.

$\rm(2)$ By \cite[Theorem A (2)]{Sh04}, there are $\Psi^*(\cdot)\in C_{\rm unif}^b(\R,\R)$ and $c^*\in\R$ such that for a.e. $\omega\in\Omega$,
$$
\lim_{t\to\infty}\frac{\xi(t;\omega)}{t}=c^*,
$$
$$
\lim_{t\to\infty}\frac{1}{t}\int_0^ t\Psi(x,\sigma_s\omega)ds=\Psi^*(x)\quad \forall \,x\in\R,
$$
and
$$
\lim_{x\to -\infty}\Psi^*(x)=1,\quad \lim_{x\to\infty} \Psi^*(x)=0.
$$

$\rm(3)$ It follows from \cite[Theorem B (2)]{Sh04}.
\end{proof}



\end{document}